\documentclass[reqno]{asl}

\usepackage{savesym}
\usepackage{color}
\usepackage{xcolor}
\usepackage{tikz}
\usepackage{tikz-cd}
\usetikzlibrary{arrows}
\usepackage{proof}
\usepackage{graphicx}
\usepackage{enumitem}
\usepackage{multicol}
\usepackage{float}

\title[Interpolation and the Exchange Rule]{Interpolation and the Exchange Rule}

\author{Wesley Fussner}
\revauthor{Wesley Fussner}

\address{Institute of Computer Science \\ 
Czech Academy of Sciences \\
Prague, Czechia}

\email{fussner@cs.cas.cz}

\author{George Metcalfe}
\revauthor{George Metcalfe}

\address{Mathematical Institute \\ 
University of Bern \\
Bern, Switzerland}

\email{george.metcalfe@unibe.ch}

\author{Simon Santschi}
\revauthor{Simon Santschi}

\address{Mathematical Institute \\ 
University of Bern \\
Bern, Switzerland}

\email{simon.santschi@unibe.ch}

\thanks{Supported by Swiss National Science Foundation grant 200021\textunderscore 215157.}

\makeatletter
\providecommand*{\leftmodels}{%
  \mathrel{%
    \mathpalette\@leftmodels\models
  }%
}
\newcommand*{\@leftmodels}[2]{%
  \reflectbox{$\m@th#1#2$}%
}
\makeatother

\newtheorem{prop}{Proposition}[section]
\newtheorem{lemma}[prop]{Lemma}

\newtheorem{mthm}{Theorem}

\theoremstyle{definition}

\newcommand{\m}[1]{{\bf {#1}}}									% Algebras
\newcommand{\lgc}[1]{\ensuremath{{\rm #1}}}						% Logics, Sans Serif
\newcommand{\cls}[1]{\ensuremath{{\sf #1}}}						% Varieties, Calligraphic
\newcommand{\pc}[1]{\ensuremath{{\sf #1}}}						% Proof Systems, Sans Serif	
\newcommand{\opr}[1]{\ensuremath{{\mathbb #1}}}					% Class operator	

\newcommand*{\pfa}[1]{\mbox{\footnotesize $#1$}} 	 

\renewcommand{\restriction}{\mathord{\upharpoonright}}

\newcommand{\set}[1]{\{ #1 \}}
\newcommand{\tuple}[1]{\ensuremath{({#1})}}

\DeclareMathOperator{\im}{im}
\newcommand{\setc}{\mathnormal{\setminus}}

\newcommand{\eq}{\approx}

\newcommand{\N}{\mathbb{N}}
\newcommand{\Z}{\mathbb{Z}}
\newcommand{\der}[1]{\vdash_{\lgc{#1}}}

\newcommand*{\Nsum}{\boxplus}
\newcommand*{\nsum}{\boxplus}
\newcommand{\rgal}[1]{{#1}^{r}}
\newcommand{\lgal}[1]{{#1}^{\ell}}
\newcommand{\Com}[2]{\m{C}^{#2}_{#1}} %summand algebra
\newcommand{\com}[2]{C^{#2}_{#1}} %universe of summand algebra

\newcommand{\Go}{\m{G}} %Gödel algebras
\newcommand{\go}{G} %universe of Gödel algebra
\newcommand{\inv}[1]{{#1}^\star} %x->e for commutative case
\newcommand{\invv}[1]{{#1}^{\star\star}} %double application

\newcommand{\Inf}[3]{\mathsf{I}(#1,#2,#3)}
\newcommand{\Fin}[3]{\mathsf{F}(#1,#2,#3)}
\newcommand{\Ee}[1]{\mathsf{E}(#1)}

\newcommand{\partition}[1]{\mathcal{P}_{#1}} 

\newcommand{\Con}[1]{\mathrm{Con}(#1)}
\newcommand{\NC}[1]{{\mathrm{NC}}(#1)}
\newcommand{\eps}{\ensuremath{\varepsilon}}
\newcommand{\ov}[1]{\overline{#1}}
\newcommand{\mdl}[1]{\models_{#1}}
\newcommand{\hm}{\opr{H}}
\newcommand{\iso}{\opr{I}}

\newcommand{\pu}{\opr{P}_U}
\newcommand{\sub}{\opr{S}}
\newcommand{\vr}{\opr{V}}
\newcommand{\hspu}{\hm\sub\pu}

\newcommand{\V}{\cls{V}}
\newcommand{\K}{\cls{K}}
\newcommand{\Kfin}{\cls{K}_{\mathrm{fin}}}
\newcommand{\Vfc}{\V^\mathrm{c}_{\mathrm{fin}}}

\newcommand{\ld}{\backslash}
\newcommand{\rd}{/}
\newcommand{\pd}{\cdot}
\newcommand{\ut}{\textrm{\textup{e}}}
\newcommand{\zr}{\textrm{\textup{f}}}
\newcommand{\mt}{\mathbin{\land}}
\newcommand{\jn}{\mathbin{\lor}}

\newcommand{\lang}{\mathcal{L}}

\renewcommand{\a}{\ensuremath{\alpha}}
\newcommand{\be}{\ensuremath{\beta}}
\newcommand{\ga}{\ensuremath{\gamma}}
\newcommand{\De}{\mathrm{\Delta}}
\newcommand{\Ga}{\mathrm{\Gamma}}

\newcommand{\Si}{\mathrm{\Sigma}}

\newcommand*{\seq}{{\vphantom{A}\Rightarrow{\vphantom{A}}}}
\newcommand*{\rseq}{\Rightarrow}

\newcommand{\idr}{(\textsc{id})}

\newcommand{\flr}{(\zr\!\rseq)}
\newcommand{\frr}{(\rseq\!\zr)}
\newcommand{\tlr}{(\ut\!\rseq)}
\newcommand{\trr}{(\rseq\!\ut)}
\newcommand{\olr}{(\lor\!\rseq)}
\newcommand{\orr}{(\rseq\!\lor)}
\newcommand{\alr}{(\land\!\rseq)}
\newcommand{\arr}{(\rseq\!\land)}
\newcommand{\pdlr}{(\pd\!\rseq)}
\newcommand{\pdrr}{(\rseq\!\pd)}

\newcommand{\ldlr}{({\ld}{\rseq})}
\newcommand{\ldrr}{(\rseq\!\!\ld)}
\newcommand{\rdlr}{(\rd\!\rseq)}
\newcommand{\rdrr}{(\rseq\!\!\rd)}
\newcommand{\cutr}{\textup{\sc (cut)}}
\newcommand{\wlr}{\textup{(i)}}
\newcommand{\wrr}{\textup{(o)}}

\newcommand{\cnr}{\textup{(c)}}

\newcommand{\exr}{\textup{(e)}}

\newcommand{\mglr}{\textup{(m)}}

\begin{document}

\begin{abstract}
It was proved by Maksimova in 1977 that exactly eight varieties of Heyting algebras have the amalgamation property, and hence exactly eight axiomatic extensions of intuitionistic propositional logic  have the deductive interpolation property. The prevalence of the deductive interpolation property for axiomatic extensions of substructural logics and the amalgamation property for varieties of pointed residuated lattices, their equivalent algebraic semantics, is far less well understood, however. Taking as our starting point a formulation of intuitionistic propositional logic as the full Lambek calculus with exchange, weakening, and contraction, we investigate the role of the exchange rule --- algebraically, the commutativity law --- in determining the scope of these properties. First, we show that there are continuum-many varieties of idempotent semilinear residuated lattices that have the amalgamation property and contain non-commutative members, and hence continuum-many axiomatic extensions of the corresponding logic that have the deductive interpolation property in which exchange is not derivable. We then show that, in contrast, exactly sixty varieties of commutative idempotent semilinear residuated lattices have the amalgamation property, and hence exactly sixty axiomatic extensions of the corresponding logic with exchange have the deductive interpolation property. From this latter result, it follows also that there are exactly sixty varieties of commutative idempotent semilinear residuated lattices whose first-order theories have a model completion.
\end{abstract}

\maketitle

%%%%%%%%%%%%%%%%%%%%%%%%%%%%%%%%%%%%%%%%%%%%%%%%

\section{Introduction}\label{s:intro}

By a remarkable result of Maksimova~\cite{Mak77}, precisely eight axiomatic extensions of intuitionistic propositional logic $\lgc{IPC}$ have the following {\em deductive interpolation property}: Given any formulas $\a,\be$ of one such logic $\lgc{L}$ satisfying $\a\der{\lgc{L}}\be$, there exists a formula $\ga$,  whose variables occur in both $\a$ and $\be$, satisfying $\a\der{\lgc{L}}\ga$ and $\ga\der{\lgc{L}}\be$.\footnote{Deductive interpolation is equivalent in this setting to {\em Craig interpolation}: Given any formulas $\a,\be$ of $\lgc{L}$ satisfying $\der{\lgc{L}}\a\to\be$, there exists a formula $\ga$, whose variables occur in both $\a$ and $\be$, satisfying $\der{\lgc{L}}\a\to\ga$ and $\der{\lgc{L}}\ga\to\be$. However, in other settings --- in particular, for axiomatic extensions of modal logics or substructural logics ---  these properties may diverge.} Maksimova's proof was essentially algebraic. First, she proved that an axiomatic extension of $\lgc{IPC}$ has deductive interpolation if and only if the associated variety (equational class) of Heyting algebras has the amalgamation property, and subsequently that there are precisely eight such varieties. This result was later strengthened by Ghilardi and Zawadowski~\cite{GZ02}, who, building on Pitts' theorem stating that $\lgc{IPC}$ has the stronger property of uniform deductive interpolation~\cite{Pit92}, proved that all eight axiomatic extensions of $\lgc{IPC}$ with deductive interpolation have this stronger property, and that the first-order theories of all eight varieties of Heyting algebras with amalgamation have a model completion. 

Similar results have been established for normal modal logics. In particular, Maksimova proved that between forty three and forty nine axiomatic extensions of $\lgc{S4}$ have deductive interpolation~\cite{Mak79}, and that continuum-many axiomatic extensions of G{\"o}del-L{\"o}b logic $\lgc{GL}$~\cite{Mak91} have the property. Less well-understood, however, is the prevalence of deductive interpolation among {\em substructural logics}. Such logics are often viewed as axiomatic extensions of the {\em full Lambek calculus} $\pc{FL}$ --- a sequent calculus presented in an algebraic language with binary operation symbols $\mt,\jn,\pd,\ld,\rd$ and constants $\ut,\zr$ --- with algebraic semantics provided by associated varieties of pointed residuated lattices (see Appendix~\ref{a:FL} and~\cite{GJKO07,MPT23}). Interpolation results have been obtained for a wide range of substructural logics (see,~e.g.,~\cite{OK85,GO06b,Mon06,KO10,Mar12,MM12,MMT14,GLT15}), but Maksimova-style descriptions of the axiomatic extensions of a logic that have deductive interpolation are known only for a few specific cases. Notably, nine axiomatic extensions of the logic R-Mingle with unit have deductive interpolation (equivalently, nine varieties of Sugihara monoids have the amalgamation property)~\cite{MM12}, and, using the fact that a variety of MV-algebras has the amalgamation property if and only if it is generated by a single totally ordered algebra~\cite{DL00}, countably infinitely many axiomatic extensions of \L ukasiewicz logic have deductive interpolation.

\begin{figure}[t]
\fbox{
{
\begin{minipage}{6.725cm}
\begin{align*}
\begin{array}{cccc}
\infer[\pfa{\exr}]{\Ga_1,\Pi_2,\Pi_1,\Ga_2\seq\De}{\Ga_1,\Pi_1,\Pi_2,\Ga_2\seq\De} & &
\infer[\pfa{\cnr}]{\Ga_1,\Pi,\Ga_2\seq\De}{\Ga_1,\Pi,\Pi,\Ga_2\seq\De}\\[.15in]
\infer[\pfa{\wlr}]{\Ga_1,\Pi,\Ga_2\seq\De}{\Ga_1,\Ga_2\seq\De} & &
\infer[\pfa{\wrr}]{\Ga_1,\Pi,\Ga_2\seq\De}{\Pi\seq}
\end{array}\\[.15in]
\infer[\pfa{\mglr}]{\Ga_1,\Pi_1,\Pi_2,\Ga_2\seq\De}{\Ga_1,\Pi_1,\Ga_2\seq\De &\Ga_1,\Pi_2,\Ga_2\seq\De} \qquad
\end{align*}
\caption{Basic structural rules}
\label{fig:structural}
\end{minipage}
}}
\end{figure}

Adding to $\pc{FL}$ the structural rules of exchange (e), weakening (i) and (o), and contraction (c) depicted in Figure~\ref{fig:structural} yields a sequent calculus for $\lgc{IPC}$, where the operations $\mt$ and $\pd$ can be identified, and likewise $\ld$ and $\rd$. Algebraically, adding (e), (i), (o), and (c) produces sequent calculi corresponding to varieties of pointed residuated lattices that are commutative, integral, bounded, and square-increasing, respectively, i.e., satisfying, respectively, the equations $x\pd y\eq y\pd x$, $x\le\ut$ and $\zr\le x$, and $x\le x\pd x$. In particular, pointed residuated lattices satisfying all these equations are term-equivalent to Heyting algebras.

Examples of axiomatic extensions of the full Lambek calculus with exchange, $\pc{FL_e}$, that have deductive interpolation are well known and abundant; indeed, continuum-many such extensions have been described in~\cite{FusSan2023}. For many years, however, it was an open question as to whether exchange is derivable in {\em every} axiomatic extension of $\pc{FL}$ that has deductive interpolation (see~\cite[Problem~5]{GLT15}). That this is not the case was shown using a counterexample in~\cite{GilJipMet2020}. Motivated by this result, we examine here the critical role played by exchange in determining the scope of deductive interpolation in substructural logics. 

As a natural starting point, we consider a sequent calculus in which exchange is not derivable that deviates as little as possible from $\lgc{IPC}$. Since exchange is already derivable in the presence of weakening and contraction, we replace (i) and (o) with a less powerful variant, the mingle rule (m) (see Figure~\ref{fig:structural}), obtaining the full Lambek calculus with mingle and contraction, $\pc{FL_{cm}}$, corresponding to the variety of idempotent pointed residuated lattices. Not only is it not possible to derive (e) in $\pc{FL_{cm}}$, a growing body of literature demonstrates many semantic similarities between $\lgc{IPC}$ and various axiomatic extensions of this logic (see., e.g.,~\cite{GR12,GR15,FG2019,JTV21,FusGal2023a,FusGal2023b}). Indeed, the  logic presented in~\cite{GilJipMet2020}  that has deductive interpolation but does not derive exchange is such an axiomatic extension. We show here that there are in fact continuum-many such logics.

Our proof is algebraic, following the approach of Maksimova. We use the fact that an axiomatic extension of $\pc{FL}$ has deductive interpolation if the associated variety of pointed residuated lattices has the amalgamation property, noting that the converse holds in the presence of a local deduction theorem for the logic, or, equivalently, the congruence extension property for the variety~\cite{MMT14}. The proof therefore amounts to exhibiting continuum-many varieties of idempotent pointed residuated lattices that have non-commutative members. In fact, these varieties have a rather special form. Their members satisfy $\ut\eq\zr$ --- and are therefore referred to simply as residuated lattices --- and are {\em semilinear}, that is, subdirect products of totally ordered algebras. Let us denote by $\pc{SemRL_{cm}}$ the axiomatic extension of $\pc{FL_{cm}}$ associated with the variety of idempotent semilinear residuated lattices, noting that this logic can also be presented as a hypersequent calculus (see,~e.g.,~\cite{MOG08}). The first main result of this paper may be stated as follows:

\newtheorem*{t:main}{Theorem~\ref{t:main}}
\begin{t:main}\
\begin{enumerate}[label = \textup{(\roman*)}]
\item Continuum-many varieties of  idempotent semilinear residuated lattices have the amalgamation property and contain non-commutative members.
\item Continuum-many axiomatic extensions of $\pc{SemRL_{cm}}$ in which exchange is not derivable have the deductive interpolation property.
\end{enumerate}
\end{t:main}

As a natural next step, we consider how the picture presented in Theorem~\ref{t:main} changes when exchange (algebraically, commutativity) is reinstated. To this end, let us denote by $\pc{SemRL_{ecm}}$ the axiomatic extension of $\pc{FL_{cm}}$ associated with the variety of commutative idempotent semilinear residuated lattices (corresponding also to the hypersequent calculus for $\pc{SemRL_{cm}}$ extended with exchange~\cite{MOG08}). The following results demonstrate that in this setting, reinstating exchange reduces the number of axiomatic extensions having deductive interpolation (algebraically, amalgamation) from continuum-many to finitely many:

\newtheorem*{t:main_finite}{Theorem~\ref{t:main_finite}}
\begin{t:main_finite}\
\begin{enumerate}[label = \textup{(\roman*)}]
\item Exactly sixty varieties of commutative idempotent semilinear residuated lattices have the amalgamation property.
\item Exactly sixty axiomatic extensions of $\pc{SemRL_{ecm}}$ have the deductive interpolation property.
\end{enumerate}
\end{t:main_finite}

Moreover, using the fact that the first-order theory of a locally finite variety has a model completion if and only if and only if the variety has the amalgamation property~\cite[p.~319, Corollary~1]{Whe76}, we obtain a complete classification of the varieties of commutative idempotent semilinear residuated lattices whose first-order theories have a model completion.

\newtheorem*{t:main_mcs}{Theorem~\ref{t:main_mcs}}
\begin{t:main_mcs}
There are exactly sixty varieties of commutative idempotent semilinear residuated lattices whose first-order theories have a model completion.
\end{t:main_mcs}

Theorems~\ref{t:main_finite} and~\ref{t:main_mcs} raise the obvious question as to whether we may cast our net still wider and obtain similar results for a weaker logic than $\pc{SemRL_{ecm}}$. With respect to omitting semilinearity, the problem of whether only finitely many varieties of commutative idempotent residuated lattices have deductive interpolation is open and appears to be quite challenging. On the other hand, we show here that we can drop the requirement that members of the variety satisfy $\ut\eq\zr$. Despite the additional combinatorial complexities involved in dealing with the constant $\zr$, we show that there are still only finitely many varieties of commutative idempotent semilinear pointed residuated lattices that have the amalgamation property, and hence finitely many axiomatic extensions of the corresponding logic that have the deductive interpolation property. We do not give the precise number of such varieties (equivalently, axiomatic extensions), but show that these number more than 12,000,000.

The paper is structured as follows. In Section~\ref{s:int and amal}, we provide the necessary background on the relationship between the deductive interpolation  property and amalgamation property, and in Section~\ref{s:IRL}, we recall some important properties of idempotent residuated lattices, presenting in particular the nested sum structure of $^\star$-involutive idempotent residuated chains. In Section~\ref{s:without_exchange}, we prove Theorem~\ref{t:main}, using the fact that here are continuum-many pairwise incomparable minimal bi-infinite words over $\{0,1\}$ to construct corresponding varieties of idempotent semilinear residuated lattices that have the amalgamation property and contain non-commutative members. In Section~\ref{s:with_exchange}, we prove Theorems~\ref{t:main_finite} and~\ref{t:main_mcs}, using the structure theory of commutative idempotent residuated chains to classify the varieties of commutative idempotent semilinear residuated lattices that have the amalgamation property, and consider also the case of (the  logic corresponding to) the variety of commutative idempotent semilinear pointed residuated lattices.

%%%%%%%%%%%%%%%%%%%%%%%%%%%%%%%%%%%%%%%%%%%%%%%%

\section{Interpolation and amalgamation}\label{s:int and amal}

In this section, we provide a brief account of the well-known relationship between deductive interpolation and amalgamation developed in~\cite{Gra75,Pig72,Bac75,CP99,MMT14}, introducing also some further concepts and tools that will play a crucial role in subsequent sections.

Let us first recall some basic notions of universal algebra, referring to~\cite{BS81} for proofs and references. Let $\lang$ be any algebraic language, assuming for convenience at least one constant symbol, and let $\m{A}$ be any {\em $\lang$-algebra}, i.e., an $\lang$-structure with universe $A$. The set of congruences of $\m{A}$ forms a lattice $\Con{\m{A}}$, ordered by inclusion.  $\m{A}$ is called {\em congruence distributive} if $\Con{\m{A}}$ is distributive, and {\em locally finite} if every finitely generated subalgebra of $\m{A}$ is finite; it is said to have the {\em congruence extension property} if for any subalgebra $\m{B}$ of $\m{A}$ and $\Theta\in\Con{\m{B}}$, there exists a $\Phi\in\Con{\m{A}}$ such that $\Phi\cap B^2=\Theta$. A class of $\lang$-algebras is said to have one of these properties when all of its members have the property.

An $\lang$-algebra $\m{A}$ is called {\em simple} if $\Con{\m{A}}$ contains only  the  least congruence $\De_A := \{\tuple{a,a}\mid a\in A\}$ and the greatest congruence $A^2$; if  $\m{A}$ is also finite and has no non-trivial proper subalgebras, it is called {\em strictly simple}. An $\lang$-algebra $\m{A}$ is said to be \emph{(finitely) subdirectly irreducible} if whenever $\m{A}$ is isomorphic to a subdirect product of a (non-empty finite) set of algebras, it is isomorphic to one of these algebras. Equivalently, $\m{A}$ is finitely subdirectly irreducible if  $\De_A$ is meet-irreducible in $\Con{\m{A}}$, and subdirectly irreducible if $\De_A$ is completely meet-irreducible in $\Con{\m{A}}$.

Let $\opr{I}$, $\opr{H}$, $\opr{S}$, $\opr{P}$, and $\opr{P}_U$ denote the class operators of taking isomorphic images, homomorphic images, subalgebras, products, and ultraproducts, respectively. A class of $\lang$-algebras is called a {\em variety} if it is closed under $\opr{H}$, $\opr{S}$, and $\opr{P}$, and the variety generated by a class of $\lang$-algebras $\K$ is $\opr{V}(\K):=\opr{H}\opr{S}\opr{P}(\K)$, the smallest variety of $\lang$-algebras containing $\K$. The set of subvarieties of a variety $\V$ (i.e., the varieties contained in $\V$) forms a lattice, ordered by inclusion.

For any set of variables $\ov{x}$, let $\m{Fm}(\ov{x})$ denote the \emph{$\lang$-formula} {\em algebra over $\ov{x}$}, writing $\a(\ov{x})$,  $\eps(\ov{x})$, or $\Si(\ov{x})$ to denote that $\ov{x}$ includes the variables of an $\lang$-formula $\a$, $\lang$-equation $\eps$, or set of $\lang$-equations $\Si$, respectively. For convenience, we assume that $\ov{x}$, $\ov{y}$, etc. denote disjoint sets, writing $\ov{x},\ov{y}$ to denote their union.

Equational consequence for a variety $\V$ is defined as follows. Given any set of $\lang$-equations $\Si\cup\{\a\eq\be\}$ containing exactly the variables in the set $\ov{x}$, 
\begin{align*}
\Si\mdl{\V}\a\eq\be\: :\Longleftrightarrow\enspace & \text{for every } \m{A} \in \V\text{ and homomorphism }h\colon \m{Fm}(\ov{x}) \to \m{A},\\
& h(\a')=h(\be')\text{ for all }\a'\eq\be'\in\Si \enspace\Longrightarrow\enspace h(\a)=h(\be).
\end{align*}
For a set of $\lang$-equations $\Si \cup \De$, we write $\Si \mdl{\V} \De$ if $\Si \mdl{\V} \eps$ for all $\eps \in \De$. 

A variety $\V$ is said to have the {\it deductive interpolation property} if for any set of $\lang$-equations $\Si(\ov{x},\ov{y}) \cup \{\eps(\ov{y},\ov{z})\}$ satisfying $\Si \mdl{\V} \eps$, there exists a set of $\lang$-equations $\Pi(\ov{y})$ satisfying $\Si \mdl{\V} \Pi$ and $\Pi \mdl{\V} \eps$. Note that if $\V$ is a variety comprising the equivalent algebraic semantics for a deductive system $\der{}$ (see Appendix~\ref{a:FL} or~\cite{BP89}), then $\V$ has the deductive interpolation property if and only $\der{}$ has this property, i.e., if for any set of $\lang$-formulas $T\cup\{\be\}$ satisfying $\Si\der{}\be$, there exists a set of formulas $T'$, whose variables occur in both $T$ and $\be$, satisfying $T\der{}T'$ and~$T'\der{}\be$.

Now let $\K$ be any class of $\lang$-algebras. A \emph{span} in $\K$ is a 5-tuple $\tuple{\m{A},\m{B},\m{C},i_B,i_C}$ consisting of $\m{A},\m{B},\m{C}\in \K$ and embeddings $i_B\colon\m{A}\to\m{B}$, $i_C\colon\m{A}\to\m{C}$. A {\em one-sided amalgam in} $\K$ of this span is a triple $\tuple{\m{D},j_B,j_C}$  consisting of some $\m{D}\in\K$, embedding $j_B\colon\m{B}\to\m{D}$, and homomorphism $j_C\colon\m{C}\to\m{D}$ such that $j_Bi_B = j_Ci_C$; it is called an {\em amalgam in }$\K$ if $j_C$ is also an embedding. The class $\K$ is said to have the {\em one-sided amalgamation property} if every span in $\K$ has a one-sided amalgam in $\K$, and the {\em amalgamation property} if every span in $\K$ has an amalgam in $\K$. 

\begin{prop}[{\cite[Corollary~3.4]{FusMet2022}}]\label{p:FusMet2022}
Let $\V$ be a variety with the congruence extension property such that the class of finitely subdirectly irreducible members of $\V$ is closed under taking subalgebras. Then $\V$ has the amalgamation property if and only if the class of finitely subdirectly irreducible members of $\V$  has the one-sided amalgamation property. 
\end{prop}

The following well-known bridge theorem relates the deductive interpolation property to the amalgamation property. 

\begin{prop}[cf.~{\cite[Theorem~22]{MMT14}}]\label{p:bridge}
Let $\V$ be a variety.
\begin{enumerate}[label = \textup{(\roman*)}]
\item If $\V$ has the amalgamation property, then it has the deductive interpolation property.
\item If $\V$ has the congruence extension property and the deductive interpolation property, then it has the amalgamation property.
\end{enumerate}
\end{prop}

\noindent
It follows also that if a variety $\V$ with the congruence extension property is the equivalent algebraic semantics of a deductive system $\der{}$, then $\der{}$ has the deductive interpolation property if and only if $\V$ has the amalgamation property.

Let us note finally, referring to~\cite{MR23} for further details and references, that the amalgamation property for a locally finite variety implies and is implied by an important model-theoretic property of the first-order theory of the variety.

\begin{prop}[{\cite[Corollary 1, p. 319]{Whe76}}]\label{p:model_completion}
Let $\V$ be a locally finite variety. The first-order theory of $\V$ has a model completion if and only if $\V$ has the amalgamation property.
\end{prop}

%%%%%%%%%%%%%%%%%%%%%%%%%%%%%%%%%%%%%%%%%%%%%%%%

\section{Idempotent residuated lattices}\label{s:IRL}

In this section, we present basic facts and structure theory for idempotent (semilinear) residuated lattices. For further details and references on residuated lattices in general, we refer to~\cite{GJKO07,MPT23}, and for idempotent semilinear residuated lattices in particular, to~\cite{Raf07,GR15,GilJipMet2020,FusGal2023a,FusGal2023b}.

A \emph{residuated lattice} is an algebraic structure $\m{A}=\tuple{A,\mt,\jn,\pd,\ld,\rd,\ut}$ such that $\tuple{A,\mt,\jn}$ is a lattice with an order defined by  $x\leq y\: :\Longleftrightarrow\: x\mt y = x$ for $x,y\in A$; $\tuple{A,\pd,\ut}$ is a monoid; and the lattice and monoid structures are linked by the following law of residuation: For all $x,y,z\in A$,
\begin{align*}
y\leq x\ld z\iff x\pd y\leq z \iff x\leq z\rd y.
\end{align*}
A \emph{pointed residuated lattice} is an expansion of a residuated lattice by an extra constant $\zr$, for which we stipulate no additional assumptions. Since the law of residuation can be replaced by a finite set of equations, pointed residuated lattices form a finitely based variety, i.e., a class of algebraic structures defined by finitely many equations. For convenience, we identify residuated lattices with pointed residuated lattices satisfying the equation $\zr\eq \ut$.

For readability, we often suppress the multiplication operation $\pd$ and write $xy$ in place of $x\pd y$, adopting the convention that multiplication has priority over $\ld,\rd$ and that the latter have priority over $\mt,\jn$, dropping parentheses accordingly. We also make frequent use of several term-definable operations in order to make our discussion more transparent and compact. In particular, we define for any pointed residuated lattice $\m{A}$ and $x\in A$,
\begin{align*}
x^r := x\ld \ut,\qquad
x^\ell := \ut\rd x,\quad\text{and}\quad
x^\star := x^\ell\mt x^r.
\end{align*}
A pointed residuated lattice is called \emph{commutative},  \emph{idempotent}, or \emph{$^\star$-involutive} if it satisfies $xy \eq yx$, $xx \eq x$, or $x^{\star\star} \eq x$, respectively; it is called \emph{totally ordered} if its underlying lattice order is total, and \emph{semilinear} if it is isomorphic to a subdirect product of totally ordered pointed residuated lattices. Since every commutative pointed residuated lattice satisfies $x\ld y \eq y\rd x$, the common value of $x\ld y$ and $y\rd x$ is in this case denoted by $x\to y$, and the common value of  $x^\star$, $x^\ell$, and $x^r$ is denoted by $x^\star$. For the sake of brevity, we also refer to a totally ordered residuated lattice as a {\em residuated chain}. 

Up to term-equivalence, an \emph{odd Sugihara monoid} is a commutative idempotent semilinear $^\star$-involutive residuated lattice, a \emph{Brouwerian algebra} is a residuated lattice satisfying $xy \eq x\mt y$, and a \emph{Heyting algebra} is a pointed residuated lattice satisfying  $xy \eq x\mt y$ and $\zr\leq x$. Brouwerian algebras and Heyting algebras are clearly both commutative and idempotent. \emph{Relative Stone algebras} and \emph{G\"odel algebras} are (again, up to term-equivalence) semilinear Brouwerian algebras and semilinear Heyting algebras, respectively.

Residuated lattices are congruence distributive. They are also {\em $\ut$-regular} in the sense that their congruences are determined by the congruence classes of $\ut$, and these congruence classes form subalgebras satisfying certain conditions. A subalgebra $\m{C}$ of a residuated lattice $\m{A}$ is called {\em convex} if whenever  $x,y\in C$, $a\in A$, and $x\leq a\leq y$, then $a\in C$, and {\em normal} if whenever $x\in C$ and $a\in A$, then $a\ld xa \mt \ut,ax\rd a\mt \ut\in A$. The convex normal subalgebras of a residuated lattice $\m{A}$ form a lattice under inclusion, denoted by $\NC{\m{A}}$, that is isomorphic to $\Con{\m{A}}$, as witnessed by the mutually inverse order-preserving maps
\begin{align*}
\Con{\m{A}}\to\NC{\m{A}};&\enspace \Theta\mapsto [\ut]_\Theta,\\
\NC{\m{A}}\to\Con{\m{A}};&\enspace \m{C}\mapsto\{(x,y)\in A^2 \mid x\ld y\mt y\ld x\in C\}.
\end{align*}
Note that, since the congruences of a pointed residuated lattice coincide with the congruences of its $\zr$-free reduct, the above characterization also applies to the pointed setting.

The following lemma provides a useful criterion for homomorphisms between residuated chains to be injective.

\begin{lemma}\label{l:inj-hom}
Let $\m{A}$ and $\m{B}$ be any residuated chains and let $a\in A$ be any subcover of $\ut$, i.e., $a<\ut$ and $a<b\le\ut$ implies $b=\ut$. Then a homomorphism $f \colon \m{A} \to \m{B}$ is injective if and only if $f(a) < f(\ut)$.
\end{lemma}

\begin{proof}
The left-to-right direction is immediate. For the converse, suppose contrapositively that $f$ is not injective. Let $\m{C}$ be the convex normal subalgebra of $\m{A}$ with universe $C:=\{x\in A\mid f(x)=\ut\}$. Since $f$ is not injective, there exists a $b\in C\setc\set{\ut}$. Moreover, we may assume that $b<\ut$, since otherwise $\ut<b$, so $b\ld \ut < \ut$ and $b\ld \ut \in C$. But then, since $a$ is a subcover of $\ut$ and $C$ is convex, $a \in C$, yielding $f(a) =\ut$.
\end{proof}

Not every pointed residuated lattice has the congruence extension property, but it is known that this property is satisfied under the assumptions of either commutativity --- reflecting the fact that every axiomatic extension of the full Lambek calculus with exchange has a local deduction theorem (see,~e.g.,~\cite{MPT23}) --- or, as recorded below for convenience, idempotency and semilinearity.

\begin{lemma}[{\cite[Corollary 4.4]{FusGal2023a}}]\label{l:CEP}
Every idempotent semilinear pointed residuated lattice has the congruence extension property.
\end{lemma}

We will make free use of basic properties of idempotent residuated lattices (chains) in performing computations, summarized in the following lemmas (for a detailed account, see, e.g., \cite[Section 3]{FusGal2023a}).

\begin{lemma}
Let ${\m A}$ be any idempotent residuated lattice. For any $x,y\in A$,
\begin{enumerate}[label = \textup{(\roman*)}]
\item $x\mt y \leq xy \leq x\jn y$;
\item if $xy \leq \ut$, then $xy = x\mt y$;
\item if $\ut \leq xy$, then $xy = x \jn y$.
\end{enumerate}
\end{lemma}

\begin{lemma} \label{l:idempotentops}\label{l:basic-properties}\label{l:galoiscon}
Let ${\m A}$ be any idempotent residuated chain. For any $x,y\in A$,
\begin{align*}
 x y = \begin{cases} 
      x &   \text{if } y\in  (x^r, x] \text{ or } y\in  [x, x^r] \\
      y &   \text{if } x\in  (y^\ell, y] \text{ or } x\in [y, y^\ell]
   \end{cases},\\
x\ld y = \begin{cases} 
      x^r \jn y &  \text{if } x\leq y \\
      x^r \mt y &  \text{if } y < x
   \end{cases}, \qquad 
y\rd x = \begin{cases} 
      x^\ell \jn y &  \text{if } x\leq y \\
      x^\ell \mt y &  \text{if } y < x
   \end{cases}.
\end{align*}
A subset of $A$ is therefore a subuniverse of $\m A$ if and only if it contains $\ut$ and is closed under the operations $x\mapsto x^\ell$ and $x\mapsto x^r$.
\end{lemma}

\begin{lemma}\label{l:homomorphism}
Let $\m{A}$ and $\m{B}$ be any idempotent residuated chains. An injective map  $h\colon A\to B$ is an embedding if and only if it is order-preserving and satisfies $h(\ut)=\ut$, $h(x^\ell)=h(x)^\ell$, and $h(x^r)=h(x)^r$, for all $x\in A$.
\end{lemma}

If an idempotent residuated chains is also $^\star$-involutive, then its normal convex subalgebras and congruences have a special form.

\begin{prop}\label{p:projective}
Let $\m{A}$ be a $^\star$-involutive idempotent residuated chain, let $\m{C}$ be a convex normal subalgebra of $\m{A}$, and let $\Theta$ be the congruence corresponding to $\m{C}$. Then $[x]_\Theta = \{x\}$ for all $x\notin C$. In particular, every quotient of $\m{A}$ is isomorphic to a subalgebra of $\m{A}$.
\end{prop}

\begin{proof}
Suppose toward a contradiction that $\tuple{x,y}\in\Theta$ with $x,y\not\in C$ and $x\neq y$. Assume further, without loss of generality, that $x<y$. From $(x,y)\in\Theta$, we have $x\ld y\mt y\ld x\in C$, and, by direct computation, $x\ld y \mt y\ld x = (x^r\jn y)\mt (y^r\mt x)$. Since $\m{A}$ is totally ordered, either $x^r\jn y\in C$ (in which case $y<x^r\in C$) or $y^r\mt x\in C$ (in which case $x>y^r\in C$). Hence either $x^r\in C$ or $y^r\in C$. Suppose that $x^r\in C$. Then also $x^{r\ell}\in C$. By \cite[Corollary~4.5]{FusGal2023b}, either $x^{r\ell}=x$ or $x^\star=x^r$. But $x=x^{r\ell}\in C$ contradicts $x\notin C$. On the other hand, $x^\star=x^r$ implies $x^\star\in C$. In this case, \cite[Lemma~4.18]{FusGal2023b} implies that $x\in C$ since $\m{A}$ is $^\star$-involutive, again a contradiction. The assumption that $y^r\in C$ similarly leads to a contradiction.

Clearly, the quotient $\m{A}/\Theta$ is obtained from ${\m A}$ by collapsing the elements in $\m{C}$ and leaving the remaining elements uncollapsed. It is then easy to see that the map $h\colon \m{A}/\Theta\to \m{A}$ defined by $h([x]_\Theta)=x$ if $x\not\in C$ and $h([x]_\Theta)=\ut$ if $x\in C$, is an embedding, so $\m{A}/\Theta$ is isomorphic to a subalgebra of $\m{A}$.
\end{proof}

We now introduce an operator for combining a family of residuated chains that is especially well behaved for $^\star$-involutive idempotent residuated chains, and will play a central role in our investigations below. 

Let us say first that a residuated chain ${\m A}$ is \emph{admissible} if $x\ld \ut, \ut\rd x \notin \{ \ut \}$ for each $x\in A\setc\{\ut\}$. Given a totally ordered set $\tuple{I,\leq}$ whose greatest element (if any) is denoted by $\top$, we say that an indexed family $({\m A}_i)_{i\in I}$ of residuated chains is \emph{admissible} if ${\m A}_i$ is admissible for all $i\in I\setc\{\top\}$.

Let $\tuple{I,\leq}$ be a non-empty totally ordered set and ${\m A}_i = (A_i,\mt_i,\jn_i,\pd_i,\ld_i,\rd_i,\ut)$ an admissible residuated chain for each $i\in I$, assuming for simplicity of notation that $A_i\cap A_j = \{\ut\}$ for $i\neq j$. We define an algebraic structure $\nsum_{i\in I} {\m A}_i$ on the union $A:=\bigcup_{i\in I} A_i$ as follows. First, we let $\leq$ be the smallest partial order on $A$ satisfying
\begin{enumerate}
\item $\leq$ extends $\leq_i$ for each $i\in I$;
\item if $i<j$, $x\in A_i$, $y\in A_j$, and $x<_i \ut$, then $x\leq y$;
\item if $i<j$, $x\in A_i$, $y\in A_j$, and $\ut<_i x$, then $y\leq x$.
\end{enumerate}
It is easy to see that $\leq$ is a total order with lattice operations $\mt$ and $\jn$. Next, for $\ast\in\{\pd,\ld,\rd\}$, we let $x \ast y := x \ast_i y$ if $x,y\in A_i$, and for $x\in A_i$, $y\in A_j$ with $i<j$, let $x\ast y := x\ast_i \ut$ and $y\ast x := \ut\ast_i x$. The resulting algebraic structure $\tuple{A,\mt,\jn,\pd,\ld,\rd,\ut}$ is denoted by $\Nsum_{\tuple{I,\leq}} {\m A}_i$, and called a \emph{nested sum} of $({\m A}_i)_{i\in I}$. Note that we can always assume that the chain $\tuple{I,\leq}$ has a top element, since we can add $\top$ to $I$ and set ${\m A}_\top$ to be a trivial algebra. Moreover, we write $\Nsum_{i=1}^k {\m A}_i$ and ${\m A}_1 \nsum {\m A}_2$ for, respectively, $I = \{1, \dots, k \}$ and $I = \{1, 2\}$ with the standard total order. We also stipulate that $\Nsum_{i=1}^k {\m A}_i$ is a trivial algebra for $k<1$.

The following structural description is fundamental to the development of ideas in subsequent sections. Parts (i) and (ii) follow from \cite[Lemma~4.18]{FusGal2023b} and part (iii) is \cite[Lemma~4.22]{FusGal2023b}.

\begin{lemma}\label{l:starinvnested}\label{l:1-gen}\
\begin{enumerate}[label = \textup{(\roman*)}]
\item Every $^\star$-involutive idempotent residuated chain is admissible.
\item The class of $^\star$-involutive idempotent residuated chains is closed under nested sums.
\item Every $^\star$-involutive idempotent residuated chain is the nested sum of its 1-generated subalgebras.
\end{enumerate}
\end{lemma}

The following result shows that embeddings between components of nested sums lift appropriately to embeddings between the nested sums themselves.

\begin{prop}\label{p:sum-map}
Let ${\m A} = \Nsum_{\tuple{I,\leq}} {\m A}_i$ and ${\m B} = \Nsum_{\tuple{J,\leq}} {\m B}_j$ be nested sums of  idempotent residuated chains and suppose that $f\colon I \to J$ is an order-embedding with $f(\top) = \top$ and $g_i \colon A_i \to B_{f(i)}$ is an embedding for each $i\in I$. Then the map $g\colon A \to B$, defined by $g(a) = g_i(a)$ for $a\in A_i$ is an embedding. In particular, the inclusion map is an embedding of ${\m B}_j$ into $\m B$ for each $j\in J$.
\end{prop}
\begin{proof}
The map $g$ is clearly well-defined, injective, and satisfies $g(\ut) = \ut$. Hence,  by Lemma~\ref{l:homomorphism}, it suffices to show that $g$ is order-preserving and preserves the operations $x\mapsto x^\ell$ and $x^r$. That $g$ preserves $^\ell$ and $^r$ follows from the fact that $g(x^\ast) = g_i(x^\ast) = g_i(x)^\ast = g(x)^\ast$, for any $i \in I$, $x\in A_i$, and $\ast\in\{\ell,r\}$, since the operations of $\m A$ and $\m B$ extend the operations on ${\m A}_i$ and ${\m B}_{f(i)}$, respectively,  and $g_i$ is a homomorphism. Finally,  suppose that $x,y \in A$  with  $x\leq y$ and $x\in A_i$, $y\in A_j$ for $i,j\in I$. There are several cases. If $i=j$, then $g(x) = g_i(x) \leq g_i(y) = g(y)$, since the order on $\m B$ extends the order on ${\m B}_{f(i)}$ and $g_i$ is order-preserving. Otherwise, $i\neq j$. If $x\leq \ut \leq y$, then clearly $g(x) = g_i(x) \leq \ut \leq g_j(y) = g(y)$. If $x\leq y \leq \ut$, then $i< j$, yielding $f(i) < f(j)$.  Moreover, $g_i(x)\leq \ut$, and $g_j(y)\leq \ut$. Hence, by the definition of the nested sum, $g(x) = g_i(x) \leq g_j(y)$. Similarly, if $\ut \leq x \leq y$, then $g(x) \leq g(y)$.
\end{proof}

In Section~\ref{s:with_exchange}, we will consider finite commutative idempotent residuated chains in detail. We fix some notation in order to ease our discussion there:

\begin{itemize}
\item For $n\in \N$, we denote by $\Go_n$ the $(n+1)$-element relative Stone algebra with universe $\go_n = \{c_n < \dots < c_1 < \ut\}$, i.e., the $(n+1)$-element commutative idempotent residuated chain with $x\pd y=x\mt y$.
\item For $m,n \in \N$, we define a commutative idempotent residuated chain $\Com{m}{n}=\tuple{\com{m}{n},\mt, \jn, \pd, \ld, \rd, \ut}$ as follows. The universe and underlying order of this algebra is given by $\com{m}{n} = \{b_m< \dots < b_0 < \ut < a_n < \dots < a_0\}$. For multiplication, we let $a_ia_j = a_{\min\{i,j\}} =a_i\jn a_j$, $b_kb_l = b_{\max\{k,l\}} = b_k\mt b_l$ for $0\leq i,j \leq n$, $0\leq k,l \leq m$, and $a_i b_k = b_k a_i = b_k$ for any $i,k$, where $\ut$ is the unit. The residual is uniquely determined by the order and the definition of $\cdot$ by setting $x\to z :=\max \{y\in\com{m}{n}\mid xy\leq z\}$.
\end{itemize}
Note that, up to isomorphism, $\Com{0}{0}$ is the three-element odd Sugihara monoid and $\Nsum_{i=1}^k \Com{0}{0}$ is the totally ordered $(2k + 1)$-element odd Sugihara monoid. In what follows we will always assume that in the nested sum $\Nsum_{i=1}^k \Com{m_i}{n_i}$ we have $\com{m_i}{n_i} = \{b_{m_i}^i < \dots < b_0^i < \ut < a_{n_i}^i < \dots a_{0}^i\}$.

\begin{lemma}[{\cite[Proposition 3.4]{GilJipMet2020}}]\label{l:skeleton}
Let $\m A$ be a finite commutative idempotent residuated chain.  Then there exists $k\in \N$ such that $\m A$ contains an isomorphic copy $\m B$ of $\Nsum_{i=1}^k \Com{0}{0}$ with $B = \{ a \in A \mid \invv{a}= a \}$. Moreover, $\m A$ is partitioned by the family of intervals $\{ B_b \}_{b \in B}$, where $B_b = \{ a \in A \mid \invv{a} = b \}$.
\end{lemma}

We call the algebra $\m B$  in the previous lemma the \emph{Sugihara skeleton} of $\m A$.

\begin{lemma}\label{l:comm-decomp}
Let $\m A$ be a finite commutative idempotent residuated chain. Then $\m A$ is isomorphic to a nested sum $(\Nsum_{i=1}^k \Com{m_i}{n_i}) \nsum \Go_p$ with $k,p\in \N$ and $m_i,n_i \in \N$ for $1\leq i \leq k$. 
\end{lemma}
\begin{proof}
Let  ${\m B} = \Nsum_{i=1}^k \Com{0}{0}$ be the Sugihara skeleton of $\m{A}$, with universe $B=\{b_0^1 < \dots < b_0^k < \ut < a_0^k<  \dots <  a_0^1 \}$. Define $m_i = \lvert B_{b_0^i} \rvert -1 $ and $n_i = \lvert B_{a_0^i} \rvert -1$ for $1 \leq i \leq k$ and $p = \lvert B_{\ut} \rvert -1$. Also, for $1 \leq i\leq k$, let  $A_i = B_{b_0^i} \cup \{ \ut \} \cup B_{a_0^i}$. Then $A_i = \{ b_{m_i}^i < \dots < b_{0}^i < \ut < a_{n_i}^i < \dots < a_0^i \}$.  It is not hard to see that the map $f_i \colon \com{m_i}{n_i} \to A_i$ defined by $f(\ut) = \ut$,  $f_i(a_j) = a_j^{i}$, $f_i(b_j) = b_j^i$ is an embedding. Moreover, for $B_{\ut} = \{ x_p < \dots < x_1 < \ut\}$ the map $f_0 \colon \go_p \to B_{\ut}$, defined by $f_0(\ut) = \ut$, and $f_(c_j) = x_j$ for $1\leq j \leq p$ is an embedding. Defining $f \colon (\Nsum_{i=1}^k \Com{m_i}{n_i}) \nsum \Go_p \to {\m A}$, by 
\begin{align*}
f(x) = \begin{cases}
f_i(x) & \text{if } x \in \com{m_i}{n_i,} \\
f_0(x) & \text{if } x \in \go_p,
\end{cases}
\end{align*}
yields the desired isomorphism. Moreover, the embedding $f$ is surjective, since $\m A$ is partitioned by $\{ B_b \}_{b\in B}$.
\end{proof}

\begin{lemma}
Suppose that $(\Nsum_{i=1}^k \Com{m_i}{n_i}) \nsum \Go_p \cong (\Nsum_{j=1}^l \Com{r_j}{s_j}) \nsum \Go_q$. Then $k=l$, $p=q$ and $r_i = m_i$, $s_i = n_i$ for all $1\leq i\leq k$.
\end{lemma}

\begin{proof}
Let ${\m A} \cong (\Nsum_{i=1}^k \Com{m_i}{n_i}) \nsum \Go_p \cong (\Nsum_{j=1}^l \Com{r_j}{s_j}) \nsum \Go_q$. Let $\m B$ be the Sugihara skeleton of $\m A$.  Note first that   $2k +1 =   \lvert \{ a \in A \mid \invv{a}  =a \} \rvert  = 2l + 1$. Hence $k = l$. Similarly we have for $B = \{b_0^1 < \dots < b_0^k < \ut < a_0^k <  \dots <  a_0^1 \}$ that $m_i = \lvert B_{b_0^i} \rvert -1 = r_i$, $n_i = \lvert B_{a_0^i} \rvert -1 = s_i$, and $p = \lvert B_{\ut} \rvert -1 = q$.
\end{proof}

We have therefore established the following structural description of finite commutative idempotent residuated chains.

\begin{prop}
Let $\m A$ be any finite commutative idempotent residuated chain. Then $\m A$ is isomorphic to a unique nested sum $(\Nsum_{i=1}^k \Com{m_i}{n_i}) \nsum \Go_p$ with $k,p\in \N$ and $m_i,n_i \in \N$ for $1\leq i \leq k$.  
\end{prop}

%%%%%%%%%%%%%%%%%%%%%%%%%%%%%%%%%%%%%%%%%%%%%%%%

\section{Interpolation without exchange}\label{s:without_exchange}

Our main aim in this section is to prove the following result:

\begin{mthm}\label{t:main}\
\begin{enumerate}[label = \textup{(\roman*)}]
\item Continuum-many varieties of  idempotent semilinear residuated lattices have the amalgamation property and contain non-commutative members.
\item Continuum-many axiomatic extensions of $\pc{SemRL_{cm}}$ in which exchange is not derivable have the deductive interpolation property.
\end{enumerate}
\end{mthm}
Theorem~\ref{t:main}(ii) follows from Theorem~\ref{t:main}(i), together with Proposition~\ref{p:bridge} and the fact that the variety of idempotent semilinear residuated lattices is an equivalent algebraic semantics for $\pc{SemRL_{cm}}$ (see Appendix~\ref{a:FL}). To prove Theorem~\ref{t:main}(i), we first recall Galatos' construction of continuum-many atoms in the lattice of subvarieties of idempotent semilinear residuated lattices~\cite{Gal2005}. Each of these atoms is generated by a single infinite non-commutative algebra $\m{A}_S$, and, as we show here, has the amalgamation property. 

Given any $S\subseteq\Z$, we let $A_S := \{a_i\mid i\in\Z\}\cup\{b_j\mid j\in\Z\}\cup\{\ut\}$, and totally order the elements of this set by stipulating $b_i< b_j<\ut<a_j<a_i$ for any $i,j\in\Z$ with $i<j$. For the multiplication, we let $\ut$ be the multiplicative unit and define for $i,j\in\Z$,  $a_ia_j := a_{\min\{i,j\}}$, $b_ib_j := b_{\min\{i,j\}}$, and 
\begin{align*}
a_ib_j := \begin{cases} 
•      a_i & \text{if }i<j\text{ or }i=j\in S \\
      b_j & \text{if }i>j\text{ or }i=j\notin S,
 \end{cases} \qquad
 b_ja_i := \begin{cases} 
      b_j & \text{if }j<i\text{ or }i=j\in S \\
      a_i & \text{if }j>i\text{ or }i=j\notin S.
   \end{cases}
\end{align*}
It is straightforward to check that this multiplication is residuated, and hence that the residual operations $\ld$ and $\rd$ satisfy for all $x,y\in A_S$,
\begin{align*}
x\ld y = \max\{z\in A_S\mid xz\leq y\},\qquad
y\rd x = \max\{z\in A_S\mid zx\leq y\}.
\end{align*}
We denote the residuated chain obtained in this way by ${\m A}_S$.

\begin{lemma}[{\cite[Corollary 5.2]{Gal2005}}]
Let $S\subseteq\Z$. Then ${\m A}_S$ is a strictly simple idempotent residuated chain, and hence generated by any element $x\in A_S\setc\{\ut\}$.
\end{lemma}

The algebras ${\m A}_S$ encode some of the dynamics of bi-infinite words. A \emph{word} over $\{0,1\}$ is a function $w\colon W\to \{0,1\}$, where $W$ is some interval of $\Z$. A word is \emph{finite} if $|W|$ is finite and \emph{bi-infinite} if $W=\Z$. We say that a finite word $v\colon W\to \{0,1\}$ is a \emph{subword} of a word $w$ if there exists an integer $k$ such that $v(i) = w(i+k)$ for all $i\in W$. For any $S\subseteq\Z$, we will often consider the characteristic function of $S$,
\begin{align*}
w_S(i) = 
\begin{cases} 
      1 & \text{if }i\in S \\
      0 & \text{if }i\notin S,
\end{cases}
\end{align*}
which is a bi-infinite word. Indeed, every bi-infinite word is of the form $w_S$ for some $S\subseteq\Z$. We define a pre-order $\sqsubseteq$ on the set of all bi-infinite words by setting $w_1\sqsubseteq w_2$ if and only if every finite subword of $w_1$ is a subword of $w_2$. For bi-infinite words $w_1,w_2$, we write $w_1\cong w_2$ if and only if $w_1\sqsubseteq w_2$ and $w_2\sqsubseteq w_1$. A bi-infinite word $w$ is said to be {\em minimal} if it is minimal with respect to the pre-order $\sqsubseteq$, i.e., if $w'\sqsubseteq w$ for some bi-infinite word $w'$, then $w'\cong w$. 

The following lemma collects some facts about the algebras ${\m A}_S$ that will be needed in what follows, summarizing portions of Lemma~5.1, Theorem~5.4, and Corollary~5.5 of \cite{Gal2005}.

\begin{lemma}[{\cite{Gal2005}}]\label{l:Gal2005}
Let $S,T\subseteq\Z$.
\begin{enumerate}[label = \textup{(\roman*)}]
\item $\vr({\m A}_S)\subseteq\vr({\m A}_{T})$ if and only if $w_S\sqsubseteq w_T$ if and only if ${\m A}_S$ embeds into the ultrapower ${\m A}_T^\N/U$ for every non-principal ultrafilter $U$ over $\N$.
\item Every non-trivial $1$-generated chain in $\vr({\m A}_S)$ is isomorphic to an algebra of the form ${\m A}_{S'}$ for some $S' \subseteq \Z$ with $w_{S'} \sqsubseteq w_S$. 
\item If $w_{S}$ is minimal, then $\vr({\m A}_S)$ is an atom in the subvariety lattice of the variety of semilinear idempotent residuated lattices.
\end{enumerate}
\end{lemma}

It follows from part~(i) of this lemma that $\vr({\m A}_S)=\vr({\m A}_{T})$ if and only if $w_S\cong w_T$. Note that the proof that there are continuum-many atoms in the subvariety lattice of idempotent semilinear residuated lattices relies on the fact that there are continuum-many pairwise incomparable minimal bi-infinite words; see \cite{Gal2005,HedMor38}.

We next assemble several technical lemmas that allow us to show that $\vr({\m A}_S)$ has the amalgamation property for any $S$ such that $w_S$ is minimal.

\begin{lemma}\label{l:Asstarinv}
Let $S\subseteq\Z$. Then ${\m A}_S$ is a $^\star$-involutive idempotent residuated chain; in particular, each ${\m A}_S$ is admissible.
\end{lemma}

\begin{proof}
Direct computation shows that, for each $i\in\Z$:
\[ a_i^\ell = \begin{cases} 
      b_i & \text{if }i\in S \\
      b_{i-1} & \text{if }i\notin S,
   \end{cases}\hspace{0.3 in}
a_i^r = \begin{cases} 
      b_i & \text{if }i\notin S \\
      b_{i-1} & \text{if }i\in S,
   \end{cases}
\]
\[ b_i^\ell = \begin{cases} 
      a_i & \text{if }i\notin S \\
      a_{i+1} & \text{if }i\in S,
   \end{cases}\hspace{0.3 in} b_i^r = \begin{cases} 
      a_i & \text{if }i\in S \\
      a_{i+1} & \text{if }i\notin S.
   \end{cases}
\]
As a consequence, $b_i^\star = a_i\mt a_{i+1}=a_{i+1}$ and $a_i^\star = b_i\mt b_{i-1}= b_{i-1}$ for all $i\in S$. It follows that $b_i^{\star\star} = a_{i+1}^\star = b_{(i+1)-1} =b_i$ and $a_i^{\star\star} = b_{i-1}^\star = a_{(i-1)+1} = a_i$. Hence $x^{\star\star} = x$ for all $x$ and ${\m A}_S$ is $^\star$-involutive. That ${\m A}_S$ is admissible is immediate from part (i) of Lemma~\ref{l:starinvnested}.
\end{proof}

Now, for $S\subseteq \Z$, let $\V_S:=\vr(\mathbf{A}_S)$ be the variety generated by ${\m A}_S$ and define  
\[
\K_S = \iso(\{{\m A}_T\mid w_T \sqsubseteq w_S\}).
\]
It follows from parts (i) and (ii) of Lemma~\ref{l:Gal2005} that $\K_S$ consists of the non-trivial $1$-generated algebras in $\V_S$. Because $\V_S$ satisfies $x^{\star\star}\eq x$ by Lemma~\ref{l:Asstarinv}, each totally ordered member of $\V_S$ is isomorphic to a nested sum of algebras from $\K_S$ by part (iii) of Lemma~\ref{l:starinvnested}.

\begin{lemma}\label{l:nested-ultrapower}
Let $\tuple{I,\leq}$ be a chain, $(\m{A}_i)_{i\in I}$ an admissible indexed family of idempotent residuated chains, and $U$ an ultrafilter over some set $X$. Then the identity map $\iota \colon \Nsum_{\tuple{I,\leq}} (\m{A}_i^X/U) \to  (\Nsum_{\tuple{I,\leq}} \m{A}_i)^X/U$, $\iota([a]_U) = [a]_U$ is an embedding. 
\end{lemma}
\begin{proof}
First note that if an idempotent residuated chain $\m A$ is admissible, then so is every ultrapower of $\m A$, since admissibility can be expressed by the quasiequations $\lgal{x}\eq \ut \Rightarrow x \eq\ut$ and $\rgal{x} \eq \ut \Rightarrow x\eq \ut$. Hence $(\m{A}_i^X/U)_{i\in I}$ is also admissible. The map $\iota$ is well defined, since for each $i\in I$ and any $a = (a_x)_{x\in X},b = (b_x)_{x\in X} \in A_i^X$,
\begin{align*}
[a]_U = [b]_U \text{ in $\m{A}_i^X/U$} & \iff \{x \in X\mid a_x = b_x \} \in U\\
& \iff [a]_U = [b]_U \text{ in $(\Nsum_{\tuple{I,\leq}} \m{A}_i)^X/U$.} 
\end{align*}
Moreover, $\iota$ is clearly injective. To see that $\iota$ is order-preserving, consider any $[a]_U, [b]_U \in\Nsum_{\tuple{I,\leq}} ({A}_i^X/U) $  with $ [a]_U \leq [b]_U$. There are three cases: either $a,b \in A_i^X$, or $a \in A_i^X$ and $b \in A_j^X$ with $i<j$, or $a \in A_i^X$ and $b \in A_j^X$ with $i>j$.

If $a,b \in A_i^X$ for some $i \in I$, then $\{x \in X\mid a_x \leq b_x \} \in U$, so also $\iota([a]_U) \leq \iota([b]_U)$. Otherwise we may assume that $[a]_U, [b]_U \neq [\ut]_U$.

If $a \in A_i^X$, $b \in A_j^X$ with $i < j$, then $[a]_U < [b]_U < [\ut]_U$ or $[a]_U < [\ut]_U < [b]_U$.
If  $[a]_U < [b]_U < [\ut]_U$, then $\{x \in X\mid a_x < \ut \}, \{x \in X\mid b_x < \ut \} \in U$. But then also $\{x \in X\mid a_x < \ut, b_x < \ut\} \in U$.  Hence, since $i<j$, by definition of the nested sum, $\{x \in X\mid a_x <\ut, b_x < \ut\}  \subseteq \{x \in X\mid a_x <  b_x \}$, i.e.,  $\{x \in X\mid a_x <  b_x \} \in U$, yielding $\iota([a]_U) \leq \iota([b]_U)$. If  $[a]_U < [\ut]_U < [b]_U$, then  $\{x \in X\mid a_x < \ut \}, \{x \in X\mid \ut < b_x \} \in U$. Hence clearly also $\iota([a]_U) < \iota([\ut]_U) <  \iota([b]_U)$.

Similarly, if $a \in A_i^X$ and $b \in A_j^X$ with $i > j$, it follows that $\iota([a]_U) \leq \iota([b]_U)$ in $(\Nsum_{\tuple{I,\leq}} \m{A}_i)^X/U$. Further, since $\iota([\ut]_U) = [\ut]_U$, $\iota$ preserves the multiplicative unit. Finally, to show that $\iota$ preserves $\lgal{\ }$ and $\rgal{\ }$, consider any $a = (a_x)_{x\in X} \in A_i$ for some $i\in I$. Then $\iota(\lgal{[a]}_U) = \iota([(\lgal{a}_x)_{x\in X}]_U) = [(\lgal{a}_x)_{x\in X}]_U = \lgal{[a]}_U$ and similarly for $\rgal{\ }$. Hence, the claim follows from Lemma~\ref{l:homomorphism}.
\end{proof}

\begin{lemma}\label{l:nsum-closure}
Let $S \subseteq \Z$.
\begin{enumerate}[label = \textup{(\roman*)}]
\item $\hspu({\m A}_S)$ is the class of totally ordered members of $\V_S$. In particular, $\hspu({\m A}_S)$ consists of the finitely subdirectly irreducible members of $\V_S$.
\item Suppose further that $w_S$ is minimal.Then $\hspu(\m{A}_S)$ is closed under nested sums. In particular, the finitely subdirectly irreducible members of $\V_S$ are exactly the nested sums of members of $\K_S$.
\end{enumerate}
\end{lemma}

\begin{proof}
(i) Each member of $\hspu({\m A}_S)$ is totally ordered by \L{}o\'{s}'s Theorem and Proposition~\ref{p:projective}. Conversely, it is easy to see that each totally ordered member of $\V_S$ is finitely subdirectly irreducible, and hence contained in $\hspu({\m A}_S)$ by Jónsson's Lemma.

(ii) First, recall that any chain in $\hspu({\m A}_S)$ is a nested sum of algebras from $\K_S$. Hence, any nested sum of chains in $\hspu({\m A}_S)$ is a nested sum of algebras from $\K_S$. Hence it suffices to show that $\hspu({\m A}_S)$ contains every nested sum of algebras from $\K_S$. Since every algebra embeds into an ultraproduct of its finitely generated subalgebras, every nested sum of algebras from $\K_S$ embeds into an ultraproduct of finite nested sums from $\K_S$. It therefore suffices to show that every finite nested sum of members of $\K_S$ is contained in $\hspu(\m{A}_S)$.  Let $\m{A}_{S_1},\dots, \m{A}_{S_n} \in \K_S$. We will show that $\Nsum_{i=1}^n \m{A}_{S_i} \in \hspu(\m{A}_S)$. 

To see this, let $U$ be a non-principal ultrafilter over $\N$. Since $\m{A}_S$ is infinite, it follows by~\cite[Theorem 1.28]{Frayne1963} that the algebra $\m{A}_S^{\N}/U \in \hspu(\m{A}_S)$ is uncountable. Hence there exists a chain $\tuple{I,\leq}$ and an indexed family $(\m{A}_{T_i})_{i\in I}$ with $\m{A}_{T_i} \in \K_S$ such that $\m{A}_S^{\N}/U  \cong \Nsum_{\tuple{I,\leq}} \m{A}_{T_i}$ and, since for each $i\in I$ the algebra $\m{A}_{T_i}$ is countable, $I$ is infinite.

Since $I$ is infinite we can consider the chain $\{1 < 2 < \dots < n \}$ as a subchain of  $\tuple{I,\leq}$ and $\Nsum_{i=1}^n \m{A}_{T_{i}} \in \hspu(\m{A}_S)$. Since $w_S$ is minimal, it follows by part~(i) of Lemma~\ref{l:Gal2005} that  $w_{S_i} \cong w_S \cong w_{T_{i}}$ for each $i\in \{1,\dots n\}$, and there exists an embedding   $f_i \colon \m{A}_{S_i} \to \m{A}_{T_{i}}^\N/U$.  Hence, by Proposition~\ref{p:sum-map}, there exists an embedding $f\colon \Nsum_{i=1}^n \m{A}_{S_i}  \to \Nsum_{i=1}^n (\m{A}_{T_{i}}^\N/U)$. But also, by Lemma~\ref{l:nested-ultrapower}, there exists an embedding $\iota \colon \Nsum_{i=1}^n (\m{A}_{T_{i}}^\N/U) \to (\Nsum_{i=1}^n \m{A}_{T_{i}})^\N/U$ and composing the two maps yields an embedding $\iota \circ f \colon  \Nsum_{i=1}^n \m{A}_{S_i} \to (\Nsum_{i=1}^n \m{A}_{T_{i}})^\N/U$, showing that $ \Nsum_{i=1}^n \m{A}_{S_i}  \in \hspu(\m{A}_S)$.
\end{proof}

\begin{lemma}\label{l:amalg}
Suppose that $S\subseteq\Z$ is such that $w_S$ is minimal. Then $\V_S$ has the amalgamation property.
\end{lemma}

\begin{proof}
By Proposition~\ref{p:FusMet2022}, it suffices to show that the class of finitely subdirectly irreducible members of $\V_S$ has the one-sided amalgamation property; this suffices because each variety $\V_S$ is congruence distributive, has the congruence extension property (Lemma~\ref{l:CEP}), and is closed under taking subalgebras (Lemma~\ref{l:nsum-closure}(i)). In fact, we show that the class of finitely subdirectly irreducible members of $\V_S$ has the amalgamation property. 

Let ${\m A}, {\m B}, {\m C}$ be finitely subdirectly irreducible members of $\V_S$, and assume without loss of generality that ${\m A}$ is a subalgebra of each of ${\m B}$ and ${\m C}$. By Lemma~\ref{l:nsum-closure}, we have that ${\m A}, {\m B}, {\m C}$ are nested sums of $1$-generated algebras in $\V_S$, and we write
\begin{align*}
{\m B} = \Nsum_{(I,\leq_I)} {\m B}_i,\qquad
{\m C} = \Nsum_{(J,\leq_J)} {\m C}_j,
\end{align*}
where each $\m{B}_i$, $\m{C}_j$ is 1-generated. Further, we may assume without loss of generality that $I\cap J$ consists of exactly those indices where ${\m A}_i =  {\m B}_i = {\m C}_i$ is a 1-generated subalgebra of ${\m A}$. Take any total order $\leq$ on $K=I\cup J$ extending $\leq_I\cup\leq_J$. For each $k\in K$, define:
\begin{enumerate}
\item ${\m D}_k = {\m A}_k$ if $k\in I\cap J$;
\item ${\m D}_k = {\m B}_k$ if $k\in I\setc J$;
\item ${\m D}_k = {\m C}_k$ if $k\in J\setc I$.
\end{enumerate}
Then ${\m D} = \Nsum_{(K,\leq)} {\m D}_k$ is an amalgam of ${\m A}, {\m B}, {\m C}$ as desired and, by Lemma~\ref{l:nsum-closure}, ${\m D} \in \V_S$. 
\end{proof}

We finally arrive at the proof of this section's main result:

\begin{proof}[Proof of Theorem~\ref{t:main}]
(i) Lemma~\ref{l:amalg} gives that $\V_S$ has the amalgamation property for each $S\subseteq\Z$ with $w_S$ minimal. From \cite{Gal2005}, there are continuum-many distinct varieties of this form. Further, each such $\V_S$ is a variety of idempotent semilinear residuated lattices that contains the non-commutative algebra $\m{A}_S$.

(ii) Immediate from (i).
\end{proof}

Among other things, Theorem~\ref{t:main} implies that there are continuum-many axiomatic extensions of $\pc{FL_{cm}}$ that have the deductive interpolation property in which exchange is not derivable. We show now that there are continuum-many axiomatic extensions in which exchange is not derivable that \emph{lack} the deductive interpolation property. 

\begin{prop}\label{p:DIP fails}\
\begin{enumerate}[label = \textup{(\roman*)}]
\item There are continuum-many varieties of idempotent semilinear residuated lattices that contain non-commutative members and lack the amalgamation property.
\item There are continuum-many axiomatic extensions of $\pc{SemRL_{cm}}$ in which exchange is not derivable and the deductive interpolation property fails.
\end{enumerate}
\end{prop}

\begin{proof}
The proof of \cite[Theorem~5.2]{FusGal2023b} exhibits a span $(\m{A},\m{B},\m{C},i_B,i_C)$ of finite idempotent residuated chains $\m{A},\m{B},\m{C}$ such that there is no idempotent semilinear residuated lattice comprising an amalgam for this span. For each $S\subseteq\Z$, let $\mathsf{W}_S$ be the variety generated by $\V_S\cup\{\m{B},\m{C}\}$. It is not hard to see, by an application of Jónsson's Lemma and the fact that there are continuum-many pairwise incomparable minimal bi-infinite words, that there are continuum-many distinct varieties of the form $\mathsf{W}_S$, where $w_S$ is minimal. Each $\mathsf{W}_S$ is a variety of idempotent semilinear residuated lattices, so the span $(\m{A},\m{B},\m{C},i_B,i_C)$ has no amalgam in $\mathsf{W}_S$. This proves part (i), and part (ii) follows as in the proof of Theorem~\ref{t:main}.
\end{proof}

%%%%%%%%%%%%%%%%%%%%%%%%%%%%%%%%%%%%%%%%%%%%%%%%

\section{Interpolation with exchange}\label{s:with_exchange}

In this section, we prove the following two results:

\begin{mthm}\label{t:main_finite}\
\begin{enumerate}[label = \textup{(\roman*)}]
\item Exactly sixty varieties of commutative idempotent semilinear residuated lattices have the amalgamation property.
\item Exactly sixty axiomatic extensions of $\pc{SemRL_{ecm}}$ have the deductive interpolation property.
\end{enumerate}
\end{mthm}

\begin{mthm}\label{t:main_mcs}
There are exactly sixty varieties of commutative idempotent semilinear residuated lattices whose first-order theories have a model completion.
\end{mthm}

Part (ii) of Theorem~\ref{t:main_finite} follows from part (i) together with Proposition~\ref{p:bridge}, using the fact that the variety of commutative idempotent semilinear residuated lattices is an equivalent algebraic semantics for $\pc{SemRL_{ecm}}$. Since commutative idempotent semilinear residuated lattices are locally finite (\cite[Corollary~3.6]{GilJipMet2020}), this result together with Proposition~\ref{p:model_completion} also implies Theorem~\ref{t:main_mcs}. The main challenge of this section is therefore to prove part (i) of Theorem~\ref{t:main_finite}.

For a class $\K$ of algebras, let $\Kfin$ denote the class of finite members of $\K$, and for a variety $\V$ of idempotent semilinear residuated lattices, let $\Vfc$ denote the class of finite totally ordered members of $\V$. The following two lemmas allow us to restrict our attention to embeddings between finite idempotent commutative residuated chains.

\begin{lemma}[{\cite[Theorem 2.3]{Graetzer2009}}]\label{l:GraQua}
Let $\V$ be any locally finite variety and let $\K \subseteq \V$ be a class of finite algebras such that $\hm\sub(\K) = \K$. Then $(\hspu(\K))_{\mathrm{fin}} = \K$. 
\end{lemma}

\begin{lemma}\label{l:ap-1ap-chains}
Let $\V$ be a locally finite variety of idempotent semilinear residuated lattices. Then $\V$ has the amalgamation property if and only if $\Vfc$ has the one-sided amalgamation property.
\end{lemma}

\begin{proof}
Note first that $\V$ has the congruence extension property by Lemma~\ref{l:CEP}. Moreover, the class of finitely subdirectly irreducible members of $\V$ consists of the totally ordered members of $\V$, and is closed under subalgebras. Hence, by \cite[Corollary~3.4]{FusMet2022}, $\V$ has the amalgamation property if and only if every span of finitely generated totally ordered members of $\V$ has a one-sided amalgam in the class of totally ordered members of $\V$. But, since $\V$ is locally finite the finitely generated totally ordered members are exactly the members of $\Vfc$.
\end{proof}

Recall that the algebras $\Go_n$ and $\Com{m}{n}$ ($n,m\in\mathbb{N}$) are defined in Section~\ref{s:IRL}.

\begin{lemma}\label{l:embed-equiv}
\hfill
\begin{enumerate}[label = \textup{(\roman*)}]
\item An injective map $f\colon \go_p \to \go_q$ is an embedding of $\Go_p$ into $\Go_q$ if and only if it is order-preserving  and $f(\ut) = \ut$.
\item An injective map $f\colon \com{m}{n} \to \com{r}{s}$ is an embedding of $\Com{m}{n}$ into $\Com{r}{s}$ if and only if it is order-preserving, $f(a_0) = a_0$, $f(b_0) = b_0$,  and $f(\ut) = \ut$.
\end{enumerate}
\end{lemma}

\begin{proof}
(i) Note that $\inv{x} = \ut$ and $\inv{y} = \ut$ for all $x \in \go_p$ and $y \in \go_q$. The claim therefore follows by Lemma~\ref{l:homomorphism}.

(ii) Note that $\{x \in \com{m}{n} \mid \invv{x} = x \} = \{b_0 < e < a_0 \}$ for all $n,m \in \N$. Hence, for every embedding $f\colon \Com{m}{n} \to \Com{r}{s}$  we must have $f(a_0) = a_0$ and $f(b_0) = b_0$. The other properties are clear. Conversely, let $f\colon \com{m}{n} \to \com{r}{s}$  be an injective and order-preserving map such that $f(a_0) = a_0$, $f(b_0) = b_0$, and $f(\ut) = \ut$. Since $f$ is order-preserving, $\ut < f(a_i)$ for all $1\leq i\leq n$. So $f(\inv{a_i}) = f(b_0) = b_0 = \inv{f(a_i)}$, and, similarly, $f(\inv{b_j}) = a_0 = \inv{f(b_j)}$ for all $1 \leq j \leq m$. Hence $f$ is a homomorphism, by Lemma~\ref{l:homomorphism}.
\end{proof}

\begin{lemma}\label{l:nsum-embed-equiv}
Suppose that ${\m A} = (\Nsum_{i=1}^k \Com{m_i}{n_i}) \nsum \Go_p$ and ${\m B} = (\Nsum_{j=1}^l \Com{r_j}{s_j}) \nsum \Go_q$. Then there is a one-to-one correspondence between embeddings $h\colon\m{A}\to\m{B}$ and triples $(f,g,\{h_i \}_{i=1}^k)$ such that $f\colon \{1,\dots, k \} \to \{ 1, \dots, l \}$ is an order-embedding, $g\colon \Go_p \to \Go_q$ is an embedding, and $h_i\colon \Com{m_i}{n_i} \to \Com{r_{f(i)}}{s_{f(i)}}$ is an embedding for each  $1\leq i \leq k$.
\end{lemma}
\begin{proof}
Let $(f,g,\{h_i \}_{i=1}^k)$ be a triple such that $f\colon \{1,\dots, k \} \to \{ 1, \dots, l \}$ is an order-embedding, $g\colon \Go_p \to \Go_q$ is an embedding, and $h_i\colon \Com{m_i}{n_i} \to \Com{r_{f(i)}}{s_{f(i)}}$ is an embedding for all $1\leq i \leq k$. By Proposition~\ref{p:sum-map}, the following map is an embedding of $\m{A}$ into $\m{B}$:
\[
h\colon A\to B;\quad
h(a) =\begin{cases}
h_i(a) & \text{if } a \in \com{m_i}{n_i} \\
g(a) & \text{if } a \in \go_p.
\end{cases}
\]
Conversely let $h\colon\m{A}\to\m{B}$ be an embedding. Then for each $1\leq i\leq k$ there is a $f(j) \in \{1,\dots, l\}$ such that $h[\com{m_i}{n_i}] \subseteq \com{r_{f(i)}}{s_{f(i)}}$. So there is a map $f\colon \{ 1 \dots, k \} \to \{1, \dots ,l \}$ that is injective, since $h$ is an embedding, and order-preserving, since otherwise the definition of the nested sum would yield a contradiction to the fact that $h$ is a homomorphism. Hence for $1\leq i \leq k$ we can define the embedding $h_i\colon \Com{m_i}{n_i} \to \Com{r_{f(i)}}{s_{f(i)}}$ to be the restriction of $h$ to $\com{m_i}{n_i}$. Note also that $h[\go_p] \subseteq \go_q$, so we can define the embedding $g\colon \Go_p \to \Go_q$ to be the restriction of $h$ to $\go_p$.  Hence we obtain the desired triple $(f,g,\{ h_i \}_{i=1}^k)$

Finally it is straightforward to check that the two constructions are inverse to each other.
\end{proof}

For $m,n,p \in \{0,1,\omega\}$, we define the following classes of commutative idempotent residuated chains:

\begin{align*}
\Ee{p} &= \iso(\{ \Go_q \mid q \leq p\}), \\
\Fin{m}{p}{n} &= \iso(\{\Com{r}{s} \nsum \Go_q, \Go_q \mid r\leq m, s\leq n, q \leq p \}), \\
 \Inf{m}{p}{n} &= \iso(\{(\Nsum_{i=1}^k \Com{r_i}{s_i}) \nsum \Go_q \mid k \in \N, r_i \leq m, s_i \leq n, q \leq p \}).
\end{align*}
Note that $\vr(\Ee{\omega})$ is the variety of relative Stone algebras,  $\vr(\Inf{0}{0}{0})$ is the variety of odd Sugihara monoids, and $\vr( \Inf{\omega}{\omega}{\omega})$ is the variety of commutative idempotent semilinear residuated lattices.

The next two lemmas follow easily from the definitions of the respective classes.

\begin{lemma}\label{l:hom-image}
Suppose that ${\m B} =  (\Nsum_{j=1}^l \Com{r_j}{s_j}) \nsum \Go_q$ is a homomorphic image of ${\m A} =  (\Nsum_{i=1}^k \Com{m_i}{n_i}) \nsum \Go_p$. Then $l\leq k$, $r_j \leq \max \{ m_i  \mid 1 \leq i \leq k\}$, $s_j \leq \max \{ n_i  \mid 1 \leq i \leq k\}$, and $q \leq \max (\{ m_i  \mid 1 \leq i \leq k\} \cup \{ p \})$ for $1\leq j \leq l$.
\end{lemma}

\begin{lemma}\label{l:HS-closure}\
\begin{enumerate}[label = \textup{(\roman*)}]
\item $\hm\sub(\Ee{p} ) = \Ee{p}$ for each $p \in \{0,1,\omega \}$.
\item $\hm\sub(\Fin{m}{p}{n}) = \Fin{m}{p}{n}$ and $\hm\sub( \Inf{m}{p}{n}) =  \Inf{m}{p}{n}$ for any $m,n,p \in \{0,1 ,\omega \}$ with $p\geq m$.
\item $\hm\sub(\Fin{m}{0}{n} \cup \Ee{p}) = \Fin{m}{0}{n} \cup \Ee{p}$ and $\hm\sub(\Inf{0}{0}{n} \cup \Ee{p}) = \Inf{0}{0}{n} \cup \Ee{p}$ for any $m,n,p \in \{0,1, \omega \}$ with $p\geq m$.
\end{enumerate}
\end{lemma}

We aim to characterize the varieties of commutative idempotent semilinear residuated lattices that have the amalgamation property. To do so, we will often argue that if $\V$ is any variety of commutative idempotent semilinear residuated lattices with the amalgamation property and $\V$ contains certain algebras, then $\V$ must also contain certain other algebras.  As a basis for this approach, we recall the following result.

\begin{lemma}[{\cite[Lemma~6.5]{GilJipMet2020}}]\label{l:AP-com-chain}
The class of finite commutative idempotent residuated chains has the amalgamation property.
\end{lemma}

We will sometimes abbreviate a span $\tuple{\m{A},\m{B},\m{C},i_B,i_B}$ by $\tuple{i_B,i_C}$, or, more explicitly, $\tuple{i_B\colon\m{A}\to\m{B},i_C\colon\m{A}\to\m{C}}$.

\begin{lemma}\label{l:amalg-for-components}\
\begin{enumerate}[label = \textup{(\roman*)}]
\item For $p \in \{0, 1, \omega\}$ every span $\tuple{i_1 \colon \Go_{q_1} \to \Go_{q_2}, i_2\colon \Go_{q_1} \to \Go_{q_3}}$ in $\Ee{p}$ has an amalgam in $\Ee{p}$.
\item For $m,n, p \in \{ 0,1, \omega \}$ every span $\tuple{i_1 \colon \Com{r_1}{s_1} \to  \Com{r_2}{s_2},i_2 \colon \Com{r_1}{s_1} \to  \Com{r_3}{s_3} }$ in $\Fin{m}{p}{n}$ has an amalgam in $\Fin{m}{p}{n}$.
\end{enumerate}
\end{lemma}

\begin{proof}
(i) If $p= 0$, then the claim is trivial. 

If $p= 1$,  up to isomorphism $\Ee{p}$ consists of the algebras $\Go_0$ and $\Go_1$, and  the unique existing embeddings are  $ \Go_0  \hookrightarrow  \Go_0$,  $ \Go_0  \hookrightarrow  \Go_1$,  $ \Go_1  \hookrightarrow  \Go_1$. So the claim follows. 

Finally suppose $p=\omega$. Then, by Lemma~\ref{l:AP-com-chain}, the span $\tuple{i_1,i_2}$ has an amalgam $(\Nsum_{i=1}^k \Com{m_i}{n_i}) \nsum \Go_{q_4}, j_1,j_2)$ in the class of finite commutative idempotent residuated chains. Hence, by Lemma~\ref{l:nsum-embed-equiv}, $j_1[\go_{q_2}] \subseteq \go_{q_4}$ and $j_2[\go_{q_3}] \subseteq \go_{q_4}$, so $(\Go_{q_4}, j_1,j_2)$ is an amalgam of the span $\tuple{i_1,i_2}$ in $\Ee{p}$.

(ii) If $m\leq 1$ and $n\leq 1$, then $r_1,r_2,r_3 \leq 1$ and $s_1,s_2,s_3 \leq 1$. Then $\Com{m}{n}$ together with the inclusion maps $j_1 \colon \Com{r_2}{s_2} \hookrightarrow \Com{m}{n}$, $j_2 \colon \Com{r_3}{s_3} \hookrightarrow \Com{m}{n}$  is an amalgam of the span $\tuple{i_1,i_2}$ in $\Fin{m}{p}{n}$.

If $m  \leq 1$ and $n = \omega$, define $f_1 = i_1\restriction_{[\ut,a_0]} \colon [\ut,a_0] \to [\ut,a_0] \subseteq \com{r_2}{s_2}$ and $f_2 = i_2\restriction_{[\ut,a_0]} \colon [\ut,a_0] \to [\ut,a_0]  \subseteq \com{r_3}{s_3}$. Let $\tuple{\{\ut < a_{k} < \dots < a_{0} \},g_1,g_2}$ be an amalgam of $\tuple{f_1,f_2}$ in the class of finite chains. Now define the maps $j_1 \colon \Com{r_2}{s_2} \to \Com{m}{k}$, $j_1(b_i) = b_i$, $j_1(a_i) = f_1(a_i)$ and  $j_2 \colon \Com{r_3}{s_3} \to \Com{m}{k}$, $j_2(b_i) = b_i$,  $j_2(a_i) = f_2(a_i)$. By Lemma~\ref{l:embed-equiv}, these maps are embeddings. Thus  $\tuple{\Com{m}{k},j_1,j_2}$ is an amalgam of the span $\tuple{i_1,i_2}$ in $\Fin{m}{p}{n}$.

The case where $m = \omega$  and $n\leq 1$ is very similar.

Finally if $m = n = \omega$, then by Lemma~\ref{l:AP-com-chain}, the span $
\tuple{i_1,i_2}$ has an amalgam $(\Nsum_{j=1}^k \Com{m_j}{n_j}) \nsum \Go_{q}, j_1,j_2)$ in the class of finite commutative idempotent residuated chains. Hence, by Lemma~\ref{l:nsum-embed-equiv}, there is a $1\leq j \leq k$ such that $\tuple{\Com{m_j}{n_j},j_1, j_2}$ is an amalgam of the span $\tuple{i_1,i_2}$ in $\Fin{m}{p}{n}$.
\end{proof}

\begin{lemma}\label{l:com-amalg-vars}
The following varieties have the amalgamation property:
\begin{enumerate}[label = \textup{(\roman*)}]
\item $\vr(\Ee{p})$ for any $p \in \{0,1,\omega\}$; 
\item $\vr(\Fin{m}{p}{n})$ for any $m,n,p \in \{ 0,1, \omega \}$ with $p \geq m$;
\item $\vr(\Inf{m}{p}{n})$ for any $m,n,p \in \{ 0,1, \omega \}$ with $p \geq m$;
\item $\vr(\Fin{m}{0}{n} \cup \Ee{p})$ for any $m,n \in \{0,1,\omega \}$, $p\in \{1,\omega \}$ with  $p \geq m$;
\item $\vr(\Inf{0}{0}{n} \cup \Ee{p})$ for any $n \in \{0,1, \omega \}$, $p\in \{1,\omega \}$.
\end{enumerate}
\end{lemma}
\begin{proof}
Let $\K$ be one of the generating sets of finite chains from (i) to (v). Then $\hm\sub(\K) = \K$, by Lemma~\ref{l:HS-closure}, and $(\hspu(\K))_{\text{fin}} = \K$, by Lemma~\ref{l:GraQua}. Hence $\K$ consists of the finite chains of $\vr(\K)$, and, by Lemma~\ref{l:ap-1ap-chains}, it suffices in each case to show that every span in $\K$ has a one-sided amalgam in~$\K$.

(i) Immediate from Lemma~\ref{l:amalg-for-components}(i).

(ii) Let $m,n,p \in \{ 0,1, \omega \}$ with $p \geq m$ and let $\tuple{i_1\colon {\m A} \to {\m B}, i_2 \colon {\m A} \to {\m C}}$ be a span in $\Fin{m}{p}{n}$. Since $\Ee{p}$ has the amalgamation property, we can assume that ${\m B}, {\m C} \notin \Ee{p}$, i.e., ${\m B} = \Com{r_2}{s_2}\nsum \Go_{q_2}$  and ${\m C} = \Com{r_3}{s_3} \nsum \Go_{q_3}$. with $r_2,r_3\leq m$, $s_2,s_3\leq n$, $q_1\leq p$. If ${\m A} = \Go_{q_1}$, then, since $\Ee{p}$ has the amalgamation property, the span $\tuple{i_1\colon \Go_{q_1} \to \Go_{q_2}, i_2 \colon \Go_{q_1} \to \Go_{q_3}}$ has an amalgam $\tuple{\Go_{q_4},f_1,f_2}$. Moreover, by Lemma~\ref{l:amalg-for-components}, the span $\tuple{g_1 \colon \Com{0}{0} \hookrightarrow  \Com{r_2}{s_2},g_2 \colon \Com{0}{0} \hookrightarrow  \Com{r_3}{s_3}}$ has an amalgam $\tuple{\Com{r_4}{s_4},h_1,h_2}$ in $\Fin{m}{p}{n}$. Now, by Lemma~\ref{l:nsum-embed-equiv}, the maps  $j_1 \colon B \to \com{r_4}{s_4} \nsum \go_{q_4}$, $j_1(x) = f_1(x)$ for $x\in \go_{q_2}$, $j_1(x) = h_1(x)$ for $x\in \com{r_2}{s_2}$,  $j_2 \colon B \to \com{r_4}{s_4} \nsum \go_{q_4}$, $j_2(x) = f_2(x)$ for $x\in \go_{q_3}$, $j_2(x) = h_2(x)$ for $x\in \com{r_3}{s_3}$ are embeddings and hence $\tuple{\Com{r_4}{s_4} \nsum \Go_{q_4}, j_1,j_2}$ is an amalgam of the span $\tuple{i_1,i_2}$ in $ \Fin{m}{p}{n}$.

If ${\m A} = \Com{r_1}{s_1} \nsum \Go_{q_1}$, then, since $\Ee{p}$ has the amalgamation property, the span $\tuple{i_1\restriction_{\go_{q_1}}\colon \Go_{q_1} \to \Go_{q_2}, i_2\restriction_{\go_{q_1}} \colon \Go_{q_1} \to \Go_{q_3}}$ has an amalgam $\tuple{\Go_{q_4},f_1,f_2}$. Moreover, by Lemma~\ref{l:amalg-for-components}, the span $\tuple{i_1\restriction_{\com{r_1}{s_1}} \colon \Com{r_1}{s_1} \to  \Com{r_2}{s_2},i_2\restriction_{\com{r_1}{s_1}} \colon \Com{r_1}{s_1} \to  \Com{r_3}{s_3}}$ has an amalgam $\tuple{\Com{r_4}{s_4},g_1,g_2}$ in $\Fin{m}{p}{n}$. Now, by Lemma~\ref{l:nsum-embed-equiv}, the maps  $j_1 \colon B \to \com{r_4}{s_4} \nsum \go_{q_4}$, $j_1(x) = f_1(x)$ for $x\in \go_{q_2}$, $j_1(x) = g_1(x)$ for $x\in \com{r_2}{s_2}$,  $j_2 \colon B \to \com{r_4}{s_4} \nsum \go_{q_4}$, $j_2(x) = f_2(x)$ for $x\in \go_{q_3}$, $j_2(x) = g_2(x)$ for $x\in \com{r_3}{s_3}$ are embeddings and hence $\tuple{\Com{r_4}{s_4} \nsum \Go_{q_4}, j_1,j_2}$ is an amalgam of the span $\tuple{i_1,i_2}$ in $ \Fin{m}{p}{n}$.

(iii) Very similar to part (ii) by first amalgamating the summands and the index sets and then using Lemma~\ref{l:nsum-embed-equiv}.

(iv) and (v) Let $m,n \in \{0,1, \omega \}$, $p\in \{1,\omega \}$.
The only spans that are not covered by parts (ii) and (iii) are  of the form $\tuple{i_1 \colon \Go_0 \to \Nsum_{i=1}^k\Com{r_i}{s_i}, i_2\colon \Go_0 \to \Go_p}$ and $\tuple{i_1'\colon \Go_0 \to \Go_p, i_2' \colon \Go_0 \to \Nsum_{i=1}^k\Com{r_i}{s_i}}$. For these define the one-sided amalgam $\tuple{\Nsum_{i=1}^k\Com{r_i}{s_i},j_1,j_2}$ with $j_1\colon  \Nsum_{i=1}^k\Com{r_i}{s_i} \to \Nsum_{i=1}^k\Com{r_i}{s_i}$, $j_1(x) = x$, $j_2 \colon \Go_p \to \Nsum_{i=1}^k\Com{r_i}{s_i}$, $j_2(x) = \ut$, and the one-sided amalgam $\tuple{\Go_p,f_1,f_2}$ with $f_1\colon  \Go_p \to \Go_p$, $f_1(x) = x$, $f_2 \colon \Nsum_{i=1}^k\Com{r_i}{s_i} \to  \Go_p$, $f_2(x) = \ut$, respectively.
\end{proof}

\begin{lemma}\label{l:contract-skeleton}
Let $\V$ be a variety of commutative idempotent semilinear residuated lattices and $\m{A} \in \Vfc$. If $\m{A} \nsum \Com{m}{n} \in \Vfc$, then $\m{A} \nsum \Go_m \in \Vfc$.
\end{lemma}

\begin{proof}
Clearly, every convex normal subalgebra of a $\Com{m}{n}$ is also a convex normal subalgebra of $\m{A} \nsum \Com{m}{n}$ by the definition of the nested sum. It is easy to see that the interval $[b_0,a_0]$ is a convex normal subalgebra of $\Com{m}{n}$ and, by direct computation, that the congruence $\Theta$ corresponding to this convex normal subalgebra satisfies $(x,y)\in\Theta$ if and only if $x=y$ or $b_0\leq x,y$, for any $x,y\in \Com{m}{n}$. Now if $x,y\in \m{A} \nsum \Com{m}{n}$ with $(x,y)\in\Theta$, then $b_0\leq x\to y \mt y\to x\leq a_0$. It follows from the definition of the nested sum that $x\to y\notin\Com{m}{n}$ if  $x\notin\Com{m}{n}$ or $y\notin\Com{m}{n}$, so also $(x,y)\in\Theta$ if and only if $x=y$ or $b_0\leq x,y\leq a_0$ for any $x,y\in \m{A} \nsum \Com{m}{n}$. Hence we obtain $\m{A} \nsum \Go_m\cong(\m{A} \nsum \Com{m}{n})/\Theta\in\Vfc$.
\end{proof}

\begin{lemma}\label{l:AP-closure}
Let $\V$ be a variety of commutative idempotent semilinear residuated lattices, let $\m{A}, \m{B} \in \Vfc$, and suppose that $\V$ has the amalgamation property.
\begin{enumerate}[label = \textup{(\roman*)}]
\item If $\m{A} \nsum \Go_2 \in \Vfc$, then $\m{A} \nsum \Go_n \in \Vfc$ for each $n\geq 1$.
\item If $\Com{0}{0} \nsum \Com{0}{0} \nsum \m{B} \in \Vfc$, then $(\Nsum_{i=1}^k \Com{0}{0} )\nsum \m{B} \in \Vfc$ for every $k\geq 1$.
\item If $\Com{2}{n} \in \Vfc$ for  some $n\in \N$, then  $\Com{m}{n},\Go_m \in \Vfc$ for every $m\in \N$.
\item If $\Com{m}{2} \in \Vfc$ for $m\in \N$, then $\Com{m}{n} \in \Vfc$ for every $n\in \N$.
\item If $\Com{m}{0}, \Com{0}{n} \in \Vfc$ for $m,n \in \Vfc$, then $\Com{m}{n}\in \Vfc$.
\item If $\m{A} \nsum \Com{0}{0} \nsum \m{B} \in \Vfc$ and $\Com{m}{n} \in \Vfc$, then $\m{A} \nsum \Com{m}{n} \nsum \m{B} \in \Vfc$.
\item If $\m{A} \nsum \Go_1 \in \Vfc$ and $\Go_p \in \Vfc$, then $\m{A} \nsum \Go_p \in \Vfc$.
\end{enumerate}
\end{lemma}
\begin{proof}
Note first that, since $\V$ has the amalgamation property, by Lemma~\ref{l:ap-1ap-chains}, $\Vfc$ has the one-sided amalgamation property. In the following we will use this without mentioning it explicitly. 

(i) Suppose that  $\m{A} \nsum \Go_2 \in \V$. We prove the claim by induction on $n$. For $n=1$ note that $\m{A} \nsum \Go_1$ is a subalgebra of $\m{A} \nsum\Go_2$. Suppose that the claim holds for $n\geq 1$, i.e., $\m{A} \nsum\Go_n\in \Vfc$. 
Consider the span of embeddings $i_1\colon \m{A} \nsum \Go_1 \to \m{A} \nsum \Go_2$, $i_1(c_1) = c_2$ and $i_1(x) = x$ for $x\in A$,  and $i_2\colon \m{A} \nsum \Go_1 \to \m{A} \nsum \Go_n$, $i_2(c_1) = c_1$ and $i_1(x) = x$ for $x\in A$. 
Then there exists an algebra $\mathbf{D} \in \Vfc$, an embedding $j_1\colon \m{A} \nsum  \Go_2 \to \m{D}$, and a homomorphism $j_2 \colon \m{A} \nsum  \Go_n \to \m{D}$ such that $j_1 \circ i_1 = j_2 \circ i_2$. 
Since $c_1$ covers $\ut$ in $\m{A} \nsum \Go_n$ and $j_2(c_1) = j_2(i_2(c_1)) = j_1(i_1(c_1)) < \ut$, by Lemma~\ref{l:inj-hom}, also $j_2$ is an embedding, so we have $j_2(c_n) < \dots < j_2(c_1) = j_1(c_2) < j_1(c_1) < \ut$. Let $S = \im(j_1) \cup \im(j_2)$. Then $\ut\in S$ and $S$ is closed under $\inv{ }$, by construction. Hence, by Lemma~\ref{l:idempotentops}, it is the universe of a subalgebra $\m S$ of $\m D$ which is clearly isomorphic to $\m{A} \nsum \Go_{n+1}$.

(ii) Suppose that $\Com{0}{0} \nsum \Com{0}{0} \nsum \m{B} \in \Vfc$. We prove the claim by induction on $k$. For $k=1$ note that $\Com{0}{0} \nsum \m{B}$ is a subalgebra of $\Com{0}{0} \nsum \Com{0}{0} \nsum \m{B}$. Suppose that the claim holds for $k\geq 1$, i.e., $(\Nsum_{i=1}^k \Com{0}{0} )\nsum \m{B} \in \Vfc$, and consider the span of embeddings $i_1 \colon \Com{0}{0} \nsum \m{B} \to \Com{0}{0} \nsum \Com{0}{0} \nsum \m{B}$, $i_1(x) = x$ for $x \in B$, $i_1(b_0) = b_0^1$, $i_1(a_0) = a_0^1$ and  $i_2 \colon \Com{0}{0} \nsum \m{B} \to (\Nsum_{i=1}^k \Com{0}{0} )\nsum \m{B}$, $i_2(x) = x$ for $x \in B$, $i_2(b_0) = b_0^k$, $i_2(a_0) = a_0^1$. Then there exists an algebra $\m{D} \in \Vfc$, an embedding $j_1 \colon \Com{0}{0} \nsum \Com{0}{0} \nsum \m{B} \to \m{D}$, and a homomorphism $j_2 \colon (\Nsum_{i=1}^k \Com{0}{0} )\nsum \m{B} \to \m{D}$ such that $j_1\circ i_1 = j_2 \circ i_2$. Let $S = \im(j_1) \cup \im(j_2)$. Then $\ut\in S$ and $S$ is closed under $\inv{ }$, by construction. Hence, by Lemma~\ref{l:idempotentops}, it is the universe of a subalgebra $\m S$ of $\m D$.  Since $j_2$ restricts to an embedding on $\Com{0}{0} \nsum \m{B}$, by Lemma~\ref{l:inj-hom}, it is also an embedding.  For each $x\in B$,
\begin{align*}
&j_2(b_0^1) < \dots < j_2(b_0^k) = j_1(b_0^1) < j_1(b_0^2) < j_1(x)  = j_2(x)\\
&j_2(x) < j_1(a_0^2) < j_1(a_0^1) = j_2(a_0^k) < \dots < j_2(a_0^1),
\end{align*}
and it is straightforward to check that  $\m S$ is isomorphic to  $(\Nsum_{i=1}^{k+1} \Com{0}{0} )\nsum \m{B}$.

(iii) Suppose that $\Com{2}{n} \in \Vfc$ for  some $n\in \N$. We prove the claim by induction on $m$. First note that, since $\Go_m$ is a homomorphic image of $\Com{m}{n}$, it suffices to show that $\Com{m}{n}\in \Vfc$ for each $m\in \N$.
For $m=1$ note that $\Com{1}{n}$ is a subalgebra of $\Com{2}{n}$. Suppose that the claim holds for $m\geq 1$, i.e., $\Com{m}{n}\in \Vfc$. We consider the span of embeddings $i_1 \colon \Com{1}{n} \to \Com{2}{n}$, $i_1(x) = x$ and $i_2 \colon \Com{1}{n} \to \Com{m}{n}$, $i_2(b_1) = b_m$, $i_2(x) = x$ for $x\neq b_1$.  Then there exists an algebra $\m{D} \in \Vfc$, an embedding $j_1 \colon \Com{2}{n}  \to \m{D}$, and a homomorphism $j_2\colon \Com{m}{n} \colon  \to \m{D}$ such that $j_1\circ i_1 = j_2 \circ i_2$. Let $S = \im(j_1) \cup \im(j_2)$. Then $\ut\in S$ and $S$ is closed under $\inv{ }$, by construction. Hence, by Lemma~\ref{l:idempotentops}, it is the universe of a subalgebra $\m{S}$ of $\m{D}$.  Since $j_2$ restricts to an embedding on $\Com{0}{0}$, by Lemma~\ref{l:inj-hom}, it is also an embedding and we get $j_2(b_1) < \dots < j_2(b_m) = j_1(b_1) < j_1(b_2)$ and $j_1(a_i) = j_2(a_i)$ for $1\leq i \leq n$. If follows that $\m{S}$ is isomorphic to $\Com{m+1}{n}$.

(iv) The proof is very similar to (iii).

(v) Suppose that $\Com{m}{0}, \Com{0}{n} \in \Vfc$ for $m,n \in \Vfc$ and consider the span of inclusions $i_1\colon \Com{0}{0} \to \Com{m}{0}$, $i_1(x) = x$ and $i_2 \colon \Com{0}{0} \to \Com{m}{0}$, $i_2(x) = x$.   Then there exists an algebra $\m{D} \in \Vfc$, an embedding $j_1 \colon \Com{m}{0}  \to \m{D}$, and a homomorphism $j_2\colon \Com{0}{n} \colon  \to \m{D}$ such that $j_1\circ i_1 = j_2 \circ i_2$. Let $S = \im(j_1) \cup \im(j_2)$. Then $\ut\in S$ and $S$ is closed under $\inv{ }$, by construction. Hence, by Lemma~\ref{l:idempotentops}, it is the universe of a subalgebra $\m S$ of $\m D$.  Since $j_2$ restricts to an embedding on $\Com{0}{0}$, by Lemma~\ref{l:inj-hom}, it is also an embedding, so $\m S$ is isomorphic to $\Com{m}{n}$.

(vi) Suppose that $\m{A} \nsum \Com{0}{0} \nsum \m{B} \in \Vfc$ and $\Com{m}{n} \in \Vfc$.
 We consider the span of inclusions $i_1 \colon \Com{0}{0} \to \m{A} \nsum \Com{0}{0} \nsum \m{B}$, $i_1(x) = x$ and $i_2 \colon \Com{0}{0} \to \Com{m}{n}$, $i_2(x) = x$. 
Then there exists an algebra $\m{D} \in \Vfc$, an embedding $j_1 \colon \m{A} \nsum \Com{0}{0} \nsum \m{B} \to \m{D}$, and a homomorphism $j_2 \colon \Com{m}{n} \to \m{D}$ such that $j_1\circ i_1 = j_2 \circ i_2$.  Let $S = \im(j_1) \cup \im(j_2)$. Then $\ut\in S$ and $S$ is closed under $\inv{ }$, by construction. Hence, by Lemma~\ref{l:idempotentops}, it is the universe of a subalgebra $\m S$ of $\m D$.
Since $j_2$ restricts to an embedding on $\Com{0}{0}$, by Lemma~\ref{l:inj-hom}, it is also an embedding and it is straightforward to check that  $\m S$ is isomorphic to $\m{A} \nsum \Com{m}{n} \nsum \m{B}$.

(vii) Suppose that  $\m{A} \nsum \Go_1 \in \Vfc$ and $\Go_p \in \Vfc$. For $p=1$ there is nothing to prove, so we may assume that $p>1$. Consider the span of embeddings $i_1\colon \Go_1 \to \Go_p$, $i_1(x) = x$ and $i_2 \colon \Go_1 \to \m{A} \nsum \Go_1$, $i_2(x) = x$.  Then there exists an algebra $\m{D} \in \Vfc$, an embedding $j_1 \colon \Go_p \to \m{D}$, and a homomorphism $j_2 \colon \m{A} \nsum \Go_1 \to \m{D}$ such that $j_1\circ i_1 = j_2 \circ i_2$.  Let $S = \im(j_1) \cup \im(j_2)$. Then $\ut\in S$ and $S$ is closed under $\inv{ }$, by construction. Hence, by Lemma~\ref{l:idempotentops}, it is the universe of a subalgebra $\m S$ of $\m D$. Moreover, since $j_2$ restricts to an embedding on $\Go_1$, by Lemma~\ref{l:inj-hom}, it is also an embedding and it is easy to check that $\m S$ is isomorphic to $\m{A} \nsum \Go_p$.
\end{proof}

We finally arrive at the proof of part (i) of Theorem~\ref{t:main_finite}, recalling that part (ii) and  Theorem~\ref{t:main_mcs} follow directly from this result.

\begin{proof}[Proof of Theorem~\ref{t:main_finite}(\textnormal{i})]
Let $\V$ be a variety of commutative idempotent semilinear residuated lattices that has the amalgamation property. It suffices to show that $\V$ is one of the varieties from Lemma~\ref{l:com-amalg-vars}. Let $\K = \Vfc$. We make a case distinction on whether algebras from the set $\{\Go_p, \Com{0}{n}, \Com{m}{0}, \Com{0}{0} \nsum \Go_p \mid m,n,p \in \{0,1, 2 \} \} \cup \{ \Com{0}{0} \nsum \Com{0}{0}  \}$ are contained in $\K$, noting that $\Go_1 \leq \Go_2$, $\Com{0}{0} \leq \Com{1}{0} \leq \Com{2}{0}$, $\Com{0}{0} \leq \Com{0}{1} \leq \Com{0}{2}$, and $\Com{0}{0} \leq \Com{0}{0} \nsum \Com{0}{0}$, and $\Com{0}{0}, \Go_p\leq \Com{0}{0} \nsum \Go_p$, and $\Go_m \in \hm(\Com{m}{0})$, by Lemma~\ref{l:contract-skeleton}.

If $\Com{0}{0} \notin \K$, then $\K \in \{ \Ee{p} \mid p \in \{0,1,\omega \} \}$, yielding $\V \in \{ \vr(\Ee{p}) \mid p \in \{0,1,\omega \} \}$. So we may assume that $\Com{0}{0} \in \K$ and let $n := \sup \{ s \in \N \mid \Com{0}{s} \in \K \}$, assuming $\sup \N = \omega$. For $\Com{0}{0} \nsum \Com{0}{0} \notin \K$, there are 11 cases:
\begin{enumerate}

\item  If  $\Com{2}{0}, \Com{0}{0} \nsum \Go_{1} \in \K$, then, by Lemma~\ref{l:AP-closure},   $\K = \Fin{\omega}{\omega}{n}$, yielding $\V = \vr(\Fin{\omega}{\omega}{n})$.

\item  If $\Com{2}{0} \notin \K$ and  $\Com{1}{0},\Com{0}{0} \nsum \Go_{2}   \in \K$, then, by Lemma~\ref{l:AP-closure},   $\K = \Fin{1}{\omega}{n}$, yielding $\V = \vr(\Fin{1}{\omega}{n})$.

\item  If $\Com{1}{0} \notin \K$ and  $\Com{0}{0} \nsum \Go_{2}   \in \K$, then, by Lemma~\ref{l:AP-closure},   $\K = \Fin{0}{\omega}{n}$, yielding $\V = \vr(\Fin{0}{\omega}{n})$.

\item  If $\Com{2}{0}, \Go_2 \notin \K$ and  $\Com{1}{0}, \Com{0}{0} \nsum \Go_{1}   \in \K$, then, by Lemma~\ref{l:AP-closure},   $\K = \Fin{1}{1}{n}$, yielding $\V = \vr(\Fin{1}{1}{n})$.

\item  If $\Com{1}{0}, \Go_2 \notin \K$ and  $\Com{0}{0} \nsum \Go_{1}   \in \K$, then, by Lemma~\ref{l:AP-closure},   $\K = \Fin{0}{1}{n}$, yielding $\V = \vr(\Fin{0}{1}{n})$.

\item  If $\Com{1}{0}, \Go_1 \notin \K$, then, by Lemma~\ref{l:AP-closure},  we get  $\K = \Fin{0}{0}{n}$, yielding $\V = \vr(\Fin{0}{0}{n})$.

\item If $\Com{0}{0} \nsum \Go_1 \notin \K$ and $ \Com{2}{0} \in \K$, then, by Lemma~\ref{l:AP-closure},  $\K = \Fin{\omega}{0}{n} \cup \Ee{\omega}$, yielding $\V = \vr(\Fin{\omega}{0}{n}\cup \Ee{\omega})$. 

\item If $\Com{0}{0} \nsum \Go_1, \Com{2}{0} \notin \K$ and $ \Com{1}{0}, \Go_2 \in \K$, then, by Lemma~\ref{l:AP-closure},  $\K = \Fin{1}{0}{n} \cup \Ee{\omega}$, yielding $\V = \vr(\Fin{1}{0}{n}\cup \Ee{\omega})$. 

\item If $\Com{0}{0} \nsum \Go_1, \Com{1}{0} \notin \K$ and $\Go_2 \in \K$, then, by Lemma~\ref{l:AP-closure},  $\K = \Fin{0}{0}{n} \cup \Ee{\omega}$, yielding $\V = \vr(\Fin{0}{0}{n}\cup \Ee{\omega})$. 

\item If $\Com{0}{0} \nsum \Go_1, \Com{2}{0}, \Go_2 \notin \K$ and $ \Com{1}{0} \in \K$, then, by Lemma~\ref{l:AP-closure},  $\K = \Fin{1}{0}{n} \cup \Ee{1}$, yielding $\V = \vr(\Fin{1}{0}{n}\cup \Ee{1})$. 

\item If $\Com{0}{0} \nsum \Go_1, \Com{1}{0}, \Go_2 \notin \K$ and $ \Go_1 \in \K$ , then, by Lemma~\ref{l:AP-closure},  $\K = \Fin{0}{0}{n} \cup \Ee{1}$, yielding $\V = \vr(\Fin{0}{0}{n}\cup \Ee{1})$. 
\end{enumerate}
For $\Com{0}{0} \nsum \Com{0}{0} \in \K$, there are another 8 cases:

\begin{enumerate}

\item  If  $\Com{2}{0} \in \K$, then, by Lemma~\ref{l:AP-closure},   $\K = \Inf{\omega}{\omega}{n}$, yielding $\V = \vr(\Inf{\omega}{\omega}{n})$.

\item  If $\Com{2}{0} \notin \K$ and  $\Com{1}{0}, \Go_{2}   \in \K$, then, by Lemma~\ref{l:AP-closure},   $\K = \Inf{1}{\omega}{n}$, yielding $\V = \vr(\Inf{1}{\omega}{n})$.

\item  If $\Com{1}{0} \notin \K$ and  $\Com{0}{0} \nsum \Go_{2}   \in \K$, then, by Lemma~\ref{l:AP-closure},   $\K = \Inf{0}{\omega}{n}$, yielding $\V = \vr(\Inf{0}{\omega}{n})$.

\item  If $\Com{2}{0}, \Go_2 \notin \K$ and  $\Com{1}{0}  \in \K$, then, by Lemma~\ref{l:AP-closure},   $\K = \Inf{1}{1}{n}$, yielding $\V = \vr(\Inf{1}{1}{n})$.

\item  If $\Com{1}{0}, \Go_2 \notin \K$ and  $\Com{0}{0} \nsum \Go_{1}   \in \K$, then, by Lemma~\ref{l:AP-closure},   $\K = \Inf{0}{1}{n}$, yielding $\V = \vr(\Inf{0}{1}{n})$.

\item  If $\Com{1}{0}, \Go_1 \notin \K$, then, by Lemma~\ref{l:AP-closure},  we get  $\K = \Inf{0}{0}{n}$, yielding $\V = \vr(\Inf{0}{0}{n})$.

\item If $\Com{0}{0} \nsum \Go_1\notin \K$ and $\Go_2 \in \K$, then, by Lemma~\ref{l:AP-closure},  $\K = \Inf{0}{0}{n} \cup \Ee{\omega}$, yielding $\V = \vr(\Inf{0}{0}{n}\cup \Ee{\omega})$. 

\item If $\Com{0}{0} \nsum \Go_1, \Go_2 \notin \K$ and $ \Go_1 \in \K$ , then, by Lemma~\ref{l:AP-closure},  $\K = \Inf{0}{0}{n} \cup \Ee{1}$, yielding $\V = \vr(\Inf{0}{0}{n}\cup \Ee{1})$.   \qed
\end{enumerate}\noqed
\end{proof}

If we drop the assumption that $\ut\eq\zr$, then there are additional cases we must consider in proving an analogue of Theorems~\ref{t:main_finite} and \ref{t:main_mcs}. Let $\V$ be a variety of commutative idempotent semilinear pointed residuated lattices. Again, it is enough to consider the class $\Vfc$ of the finite totally ordered members of $\V$, since the results we used above also work in the pointed case as well. For each algebra $\tuple{A, \mt, \jn, \cdot, \to, \ut, \zr} =\m{A} \in \Vfc$ one of the following conditions holds:
\begin{enumerate}
    \item $\zr$ is on the Sugihara skeleton of $\m{A}$
    		\begin{enumerate}
   			 \item and $\zr = \ut$ in $\m{A}$, or
			 \item and $\zr < \ut$ in $\m{A}$, or
			 \item and $\zr > \ut$ in $\m{A}$.
    		 \end{enumerate}
			\item $\zr$ is not on the Sugihara skeleton of $\m{A}$
		\begin{enumerate}
			\item and $\zr < \ut$ and $\inv{\zr} = \ut$ in $\m{A}$, or
			\item and $\zr < \ut$ and $\inv{\zr} > \ut$ in $\m{A}$, or
			\item and $\zr >\ut$  in $\m{A}$.
		\end{enumerate}
\end{enumerate} 
For $i \in  \{1,2\}$ and $x \in \{\text{a},\text{b},\text{c}\}$ we define the class 
\[
\V_{i(x)} = \{ \m{A} \in \Vfc \mid \m{A} \text{ satisfies condition } i.(x) \}.
\]
Note that $\partition{\V} = \{\V_{i(x)}\mid i \in  \{1,2\},\, x \in \{\text{a},\text{b},\text{c}\} \}$  forms a partition of $\Vfc$ such that no algebra from one class embeds into an algebra from another class, since the conditions are preserved by embeddings. Thus, to check whether $\V$ has the amalgamation property, it is enough to check whether each of the classes in $\partition{\V}$ has the one-sided amalgamation property. In particular, to show that there are only finitely many varieties of commutative idempotent semilinear pointed residuated lattices with the amalgamation property it is enough to show that there are only finitely many possibilities for $\partition{\V}$ such that each of the classes in $\partition{\V}$ has the one-sided amalgamation property.  Let $\m{A}$ be a pointed residuated lattice and $\m{B}$ a residuated lattice. We will write  $\m{A} \simeq \m{B}$ if $\m{B}$ is (isomorphic to) the $\zr$-free reduct of $\m{A}$. 

For $i \in \{1, 2 \}$ and $x \in \{\text{a},\text{b},\text{c} \}$ we define $\m{S}_{i(x)}$ to be the smallest idempotent commutative chain satisfying $i.(x)$. Note that $\m{S}_{i(x)}$ is well-defined and for each algebra $\m{A}$ satisfying $i.(x)$ we have that $\m{S}_{i(x)}$ is isomorphic to the subalgebra of $\m{A}$ generated by $\zr$. In fact, the algebra $\m{S}_{1(\text{a})}$ is just a trivial algebra, $\m{S}_{1(\text{b})} \simeq \Com{0}{0}$ with $\zr = b_0$, and  $\m{S}_{1(\text{c})} \simeq \Com{0}{0}$ with $\zr = a_0$. Moreover, $\m{S}_{2(\text{a})} \simeq \Go_1$ with $\zr \simeq c_1$, $\m{S}_{2(\text{b})} \simeq \Com{1}{0}$ with $\zr = b_1$, and  $\m{S}_{2(\text{c})} \simeq \Com{0}{1}$ with $\zr = a_1$. 

In what follows, the nested sums $\m{A}\nsum \m{B}$ and $\m{B}\nsum \m{A}$  of an idempotent pointed residuated chain $\m{A}$ with an idempotent residuated chain $\m{B}$ are understood in the obvious way, i.e., as the nested sum of the $\zr$-free reducts of $\m{A}$ with $\m{B}$ with constant $\zr$ from $\m{A}$. 

\begin{lemma}\label{l:AP-closure-pointed}
Let $\V$ be a variety of commutative idempotent semilinear pointed residuated lattices with the amalgamation property.
\begin{enumerate}[label = \textup{(\roman*)}]
\item Let $i \in \{1,2\}$, $x \in \{\textup{b},\textup{c}\}$. If $\m{A}$ is a commutative idempotent semilinear pointed residuated chain that satisfies condition $i.(x)$, then $\m{A}$ is of the form $\m{A}_1 \nsum \m{A}_2 \nsum \m{A}_3$ with $\zr \in A_2$ and  $\m{A}_2 \simeq \Com{m}{n}$. 
Moreover, $\m{A} \in \V_{i(x)}$ if and only if $\m{A}_1 \nsum \m{S}_{i(x)} \in \V_{i(x)}$,  $\m{A}_2 \in \V_{i(x)}$, and $\m{S}_{i(x)} \nsum \m{A}_3 \in \V_{i(x)}$.
\item If $\m{A}$ is a commutative idempotent semilinear pointed residuated chain that satisfies condition $2.({\textup{a}})$, then $\m{A}$ is of the form $\m{A}_1 \nsum \m{A}_2$ with $\zr \in A_2$ and $\m{A}_2 \simeq \Go_p$ for some $p>0$. Moreover, $\m{A} \in \V_{2(\textup{a})}$ if and only if $\m{A}_1 \nsum \m{S}_{2(\textup{a})} \in \V_{2(\textup{a})}$ and $\m{A}_2 \in \V_{2( \textup{a} )}$.
\end{enumerate}
\end{lemma}

\begin{proof}
The first part of (i) and (ii) follows from Lemma~\ref{l:comm-decomp} applied to the $\zr$-free reducts.
For the second part the left-to-right directions are trivial and the proof of the right-to-left directions is very similar to the proof of Lemma~\ref{l:AP-closure}.
\end{proof}

By $\pc{SemFL_{ecm}}$ we mean the axiomatic extension of $\pc{FL_{ecm}}$ associated with the variety of commutative semilinear idempotent pointed residuated lattices.

\begin{prop}\
\begin{enumerate}[label = \textup{(\roman*)}]
\item There are only finitely many varieties of commutative semilinear idempotent pointed residuated lattices that have the amalgamation property.
\item There are only finitely many axiomatic extensions of $\pc{SemFL_{ecm}}$ with the deductive interpolation property.
\end{enumerate}
\end{prop}

\begin{proof}[Proof sketch]
(i) Let $\V$ be a variety of commutative semilinear idempotent pointed residuated lattices that has the amalgamation property. By the above discussion it is enough to show that there are only finitely many possibilities for $\partition{\V} = \{ \V_{i(x)}\mid  i \in \{1,2\},\, x \in \{\text{a},\text{b},\text{c} \} \}$ such that each class has the one-sided amalgamation property. For $\V_{1(\text{a})}$ we have already classified all the 60 possibilities. 

For $i \in  \{1,2\}$ and $x \in \{\text{b},\text{c} \}$, by Lemma~\ref{l:AP-closure-pointed}(i), it is enough to separately classify which algebras of the form $\m{A}_1 \nsum \m{S}_{i(x)}$  $\m{A}_2\simeq \Com{m}{n}$, and $\m{S}_{i(x)} \nsum \m{A}_3$ are contained in $\V_{i(x)}$. First we notice that, by adapting the proof of Lemma~\ref{l:AP-closure}, we get that for the class $\{  \m{A}  \in \V_{i(x)}\mid \m{A} \cong \m{S}_{i(x)} \nsum \m{A}_3 \}$ there are essentially the same 60 possibilities as in the case where $\ut\eq\zr$, and similarly for  the class $\{ \m{A} \in  \V_{i(x)}:  \m{A} \cong \m{A}_1 \nsum \m{S}_{i(x)} \} $.
Moreover, it can be shown by using similar\footnote{Roughly speaking, there are at most three places where extra points can be added to $\m{S}_{i(x)}$: Below $\zr$ in the `same block', above $\zr$ in the `same block' (if $\zr$ is not on the Sugihara skeleton), and below $\inv{\zr}$ in the block of $\inv{\zr}$. In each of these places the possibilities are either add no point, at most one point, or two and more points, similarly to Lemma~\ref{l:AP-closure}(i).} closure properties as in Lemma~\ref{l:AP-closure} that there are only finitely many possibilities for the class $\{ \m{A}_2 \in \V_{i(x)} \mid \m{A}_2 \simeq \Com{m}{n} \}$. 
Thus for $i \in \{1,2\}$ and $x \in \{\text{b},\text{c} \}$  there are only finitely many possibilities for $\V_{i(x)}$.

For $i= 2$ and $x=\text{a}$ we note that, by Lemma~\ref{l:AP-closure-pointed}(ii), it is enough to separately classify which algebras of the form $\m{A}_1 \nsum \m{S}_{2(\text{a})}$ and which algebras $\m{A}_2$ with $\m{A}_2 \simeq \Go_p$ and $p>1$ are contained in $\V_{2(\text{a})}$. As above it follows that for the class $\{\m{A} \in \V_{2(\text{a})} \mid  \m{A} \cong \m{A}_1 \nsum \m{S}_{2(\text{a})} \}$, there are at most 60 possibilities and for the class $\{ \m{A} \in \V_{2(\text{a})}  \mid \m{A} \simeq \Go_p \text{ for some } p>1 \}$ it can again be shown using similar closure properties as in Lemma~\ref{l:AP-closure} that there are only finitely many possibilities. Hence there are only finitely may possibilities for $\V_{2(\text{a})}$. 

Part (ii) then follows because $\pc{SemFL_{ecm}}$ is algebraized by the variety of commutative idempotent semilinear pointed residuated lattices.
\end{proof}

If we let $\V_{1(\text{a})} = \Ee{\omega}$ with $\zr$ interpreted as $\ut$ and let $\V_{2(\text{a})}$ be the class of all commutative idempotent pointed residuated chains $\m{A}$ with $\m{A} \simeq \Go_p$ for some $p\in \N$ that satisfy $2.(\text{a})$, then there are at least 60 choices for $\V_{1(\text{b})}$, 60 choices for $\V_{1(\text{c})}$,  60 choices for $\V_{2(\text{b})}$, and 60 choices for $\V_{2(\text{c})}$ such that the obtained classes together form $\partition{\V}$ for some variety $\V$ that has the amalgamation property. That there are 60 choices in each case follows by considering the cases where all chains in the respective class are of the form $\m{S}_{i(x)} \nsum \m{A}$, where at least 60 suitable choices for classes of  algebras $\m{A}$ may be found, as in the case for $\ut\eq\zr$.
That these choices actually form $\partition{\V}$ for some variety $\V$ follows from the fact that a homomorphic image of an algebra of the form $\m{S}_{i(x)} \nsum \m{A}$ for $x\in \{\text{b},\text{c} \}$ will always be either contained again in $\V_{i(x)}$, in $\V_{1(\text{a})}$, or in $\V_{2(\text{a})}$. Finally, to show that these classes all have the one-sided amalgamation property follows analogously as in the case where $\ut\eq\zr$. Thus we obtain the following result:

\begin{prop}\
\begin{enumerate}[label = \textup{(\roman*)}]
\item There are at least $60^4 > 12,000,000$  varieties of commutative idempotent semilinear pointed residuated lattices that have the amalgamation property.
\item There are at least $60^4 > 12,000,000$ axiomatic extensions of $\pc{SemFL_{ecm}}$ with the deductive interpolation property
\end{enumerate}
\end{prop}

%%%%%%%%%%%%%%%%%%%%%%%%%%%%%%%%%%%%%%%%%%%%%%%%%%%%%%%%%%%%%%%%%%%%%%%

\appendix
\section{The full Lambek calculus}\label{a:FL}

\begin{figure}[t]
\centerline{
\fbox{
\begin{minipage}{10cm}
\begin{align*}
\begin{array}{lcl}
\text{Identity Axioms} &  &\text{Cut Rule}\\[.1in]
\infer[\pfa{\idr}]{\a\seq\a}{} & &\infer[\pfa{\cutr}]{\Ga_1,\Ga_2,\Ga_3\seq\De}{\Ga_2\seq\a &\Ga_1,\a,\Ga_3\seq\De}\\[.2in]
\text{Left Operation Rules} & &\text{Right Operation Rules}\\[.1in]
\infer[\pfa{\tlr}]{\Ga_1,\ut,\Ga_2\seq\De}{\Ga_1,\Ga_2\seq\De} &  &\infer[\pfa{\trr}]{\seq\ut}{}\\[.15in]
\infer[\pfa{\flr}]{\zr\seq}{} &  &\infer[\pfa{\frr}]{\Ga\seq\zr}{\Ga\seq}\\[.15in]
\infer[\pfa{\rdlr}]{\Ga_1,\be\rd\a,\Ga_2,\Ga_3\seq\De}{\Ga_2\seq\a &\Ga_1,\be,\Ga_3\seq\De} &  & 
\infer[\pfa{\rdrr}]{\Ga\seq\be\rd\a}{\Ga,\a\seq\be}\\[.15in]
\infer[\pfa{\ldlr}]{\Ga_1,\Ga_2,\a\ld\be,\Ga_3\seq\De}{\Ga_2\seq\a &\Ga_1,\be,\Ga_3\seq\De} &  & 
\infer[\pfa{\ldrr}]{\Ga\seq\a\ld\be}{\a,\Ga\seq\be}\\[.15in]
\infer[\pfa{\pdlr}]{\Ga_1,\a\pd\be,\Ga_2\seq\De}{\Ga_1,\a,\be,\Ga_2\seq\De} &  &
\infer[\pfa{\pdrr}]{\Ga_1,\Ga_2\seq\a\pd\be}{\Ga_1\seq\a &\Ga_2\seq\be}\\[.15in]
\infer[\pfa{\alr_1}]{\Ga_1,\a\mt\be,\Ga_2\seq\De}{\Ga_1,\a,\Ga_2\seq\De} & & 
\infer[\pfa{\orr_1}]{\Ga\seq\a\jn\be}{\Ga\seq\a}\\[.15in]
\infer[\pfa{\alr_2}]{\Ga_1,\a\mt\be,\Ga_2\seq\De}{\Ga_1,\be,\Ga_2\seq\De} & & 
\infer[\pfa{\orr_2}]{\Ga\seq\a\jn\be}{\Ga\seq\be}\\[.15in]
\infer[\pfa{\olr}]{\Ga_1,\a\jn\be,\Ga_2\seq\De}{\Ga_1,\a,\Ga_2\seq\De &\Ga_1,\be,\Ga_2\seq\De}
 & &\infer[\pfa{\arr}]{\Ga\seq\a\mt\be}{\Ga\seq\a &\Ga\seq\be}
\end{array}
\end{align*}
\caption{The full Lambek calculus $\pc{FL}$}
\label{fig:FL}
\end{minipage}}}
\end{figure}

We restrict our attention in this appendix to formulas constructed over a countably infinite set of variables using the binary operation symbols $\mt,\jn,\pd,\ld,\rd$ and constants $\ut,\zr$, and define a {\em sequent} to be an ordered pair consisting of a finite sequence $\Ga$ of formulas and a sequence $\De$ of at most one formula, denoted by $\Ga\seq\De$. Concatenations of finite sequences are denoted using commas, and empty sequences by an empty space. The {\em full Lambek calculus} $\pc{FL}$ displayed in Figure~\ref{fig:FL} consists of a set of sequent rules presented schematically, where the (indexed) symbols $\Ga$ and $\Pi$ stand for arbitrary finite sequences of formulas, $\De$ for an arbitrary sequence of at most one formula, and $\a$ and $\be$ for arbitrary formulas. Further well-studied sequent calculi are obtained by adding the basic structural rules of exchange (e), weakening (i) and (o), contraction (c), and mingle (m), depicted in Figure~\ref{fig:structural}.

A {\em derivation} in a sequent calculus $\pc{C}$ of a sequent $S$ from a set of sequents $R$ is a finite tree of sequents with root $S$ such that each node is either a member of $R$ or constructed from its parent nodes using a rule of $\pc{C}$. Given a set of formulas $T\cup\{\a\}$, we write $T\der{\pc{C}}\a$ if there exists a derivation of $\seq\a$ from the set of sequents $\{\seq\be\mid\be\in T\}$. In the case of $\pc{FL}$ and its extensions with basic structural rules, the relation $\der{\pc{C}}$ is a {\em deductive system}: a relation $\der{}$ between sets of formulas and formulas that satisfies, for any set of formulas $T\cup T'\cup\{\a\}$:
\begin{enumerate}
\item if $\a\in T$, then $T\der{}\a$;
\item if $T\der{}\a$ and $T'\der{}\be$ for every $\be\in T$, then $T'\der{}\a$;
\item if $T\der{}\a$, then $\sigma[T]\der{}\sigma(\a)$ for any substitution $\sigma$;
\item if $T\der{}\a$, then there exists a finite $T_0\subseteq T$ such that $T_0\der{}\a$.
\end{enumerate}
Moreover, it is common, when no confusion may occur, to denote by $\pc{C}$ the deductive system $\der{\pc{C}}$ induced by a sequent calculus $\pc{C}$. 

Given any deductive system $\der{}$ and set of formulas $W$, there exists a smallest deductive system $\der{}'$ extending $\der{}$ such that $\der{}'\a$ for each $\a\in W$, referred to as an {\em axiomatic extension} of $\der{}$. Every axiomatic extension of $\pc{FL}$ is  algebraizable in the sense of~\cite{BP89} with respect to a variety of pointed residuated lattices. For any formula $\a$ and  equation $\be\eq\ga$, let \begin{align*}
\tau(\a):=\alpha\mt\ut\eq\ut, & & \rho(\be\eq\ga):=(\be\rd\ga)\mt(\ga\rd\be).
\end{align*} 
The `transformers' $\tau$ and $\rho$ translate between axiomatic extensions of $\pc{FL}$ and equational consequence relations of varieties of pointed residuated lattices. For any axiomatic extension $\der{\pc{L}}$ of $\pc{FL}$, let $\V_\pc{L}$ be the variety of pointed residuated lattices satisfying $\{\tau(\a)\mid~\der{\pc{L}}\a\}$, and for any variety of pointed residuated lattices $\V$, let $\der{\pc{L}_\V}$ be the axiomatic extension of $\pc{FL}$ such that $\der{\pc{L}}\rho(\be\eq\ga)$ whenever $\mdl{\V}\be\eq\ga$. Then, for every set of formulas $T\cup\{\a\}$ and set of equations $\Si\cup \set{\be\eq\ga}$ in the language of pointed residuated lattices,
\begin{align*}
T \der{\pc{L}}\a &\iff \tau[T] \mdl{\pc{L}_\V} \tau(\a) \\
\Si \models_{\V_\pc{L}} \be \eq \ga &\iff \rho[\Si] \der{\pc{L}_\V} \rho(\be\eq\ga) \\
\a & \  \dashv\der{\pc{L}_\V} \rho[\tau(\alpha)] \\
\be \eq \ga & \  \leftmodels \mdl{\V_\pc{L}} \tau[\rho(\be\eq\ga)],
\end{align*}
where $T_1\der{}T_2$ abbreviates $T_1\der{}\be$ for each $\be\in T_2$, $T_1\dashv\der{\pc{L}}T_2$ abbreviates $T_1\der{\pc{L}}T_2$ and $T_2\der{\pc{L}}T_1$, and likewise for $ \models_{\V}$.

The maps $\der{\pc{L}}\mapsto\V_\pc{L}$ and $\V\mapsto\,\der{\pc{L}_\V}$ are mutually inverse, and give a dual lattice isomorphism between the lattice of axiomatic extensions of $\pc{FL}$ and the lattice of varieties of pointed residuated lattices. In particular, the axiomatic extensions of $\der{\pc{L}}$ are in bijective correspondence with subvarieties of $\V_\pc{L}$ for any axiomatic extension $\der{\pc{L}}$ of $\pc{FL}$.

\bibliographystyle{asl}
\bibliography{bibliography}

\end{document}